\let\accentvec\vec

\documentclass[runningheads,a4paper]{llncs}

\usepackage{fullpage}

\let\vec\accentvec

\usepackage{amsmath,amsfonts,amssymb,amscd,amsmath}
\usepackage{graphicx}
\usepackage{listings}
\usepackage{enumerate}
\usepackage{colonequals}
\usepackage{microtype}
\usepackage{enumitem}
\usepackage{color}
\usepackage{array}

\usepackage{url}
\usepackage[breaklinks]{hyperref}

\urldef{\mailsa}\path|{balko,cibulka}@kam.mff.cuni.cz|

\definecolor{modra}{rgb}{0,0,.8}
\definecolor{piros}{rgb}{.8,0,0}
\definecolor{zelena}{rgb}{0,.5,0}

\newcommand{\Zbb}{\mathbb{Z}}
\DeclareMathOperator{\vol}{vol}
\DeclareMathOperator{\Sh}{Sh}
\DeclareMathOperator{\lin}{lin}
\DeclareMathOperator{\aff}{aff}
\DeclareMathOperator{\inc}{I}

\newcolumntype{C}[1]{>{\centering\let\newline\\\arraybackslash\hspace{0pt}}m{#1}}

\begin{document}

\mainmatter 

\title{Covering lattice points by subspaces and \\ counting point-hyperplane incidences\thanks{The first and the third author acknowledge the support of the grants GA\v{C}R 14-14179S of Czech Science Foundation, ERC Advanced Research Grant no 267165 (DISCONV), and GAUK 690214 of the Grant Agency of the Charles University. The first author is also supported by the grant SVV--2016--260332. }}

\titlerunning{Covering lattice points by subspaces and counting point-hyperplane incidences
}

\author{Martin Balko\inst{1,2} \and Josef Cibulka\inst{1} \and Pavel Valtr\inst{1,2}}

\authorrunning{M. Balko, J. Cibulka, and P. Valtr}

\institute{Department of Applied Mathematics, \\
Faculty of Mathematics and Physics, Charles University, \\
Malostransk\'e n\'am.~25, 118 00~ Praha 1, Czech Republic\\
\mailsa
\and
Alfr\'{e}d R\'{e}nyi Institute of Mathematics,\\ Hungarian Academy of Sciences, Budapest, Hungary
}

\toctitle{Lecture Notes in Computer Science}
\tocauthor{M. Balko, J. Cibulka, and P. Valtr}
\maketitle

\begin{abstract}
Let $d$ and $k$ be integers with $1 \leq k \leq d-1$.
Let $\Lambda$ be a $d$-dimensional lattice and let $K$ be a $d$-dimensional compact convex body symmetric about the origin.
We provide estimates for the minimum number of $k$-dimensional linear subspaces needed to cover all points in $\Lambda \cap K$.
In particular, our results imply that the minimum number of $k$-dimensional linear subspaces  needed to cover the $d$-dimensional $n \times \cdots \times n$ grid is at least $\Omega(n^{d(d-k)/(d-1)-\varepsilon})$ and at most $O(n^{d(d-k)/(d-1)})$, where $\varepsilon>0$ is an arbitrarily small constant.
This nearly settles a problem mentioned in the book of Brass, Moser, and Pach~\cite{braMoPa05}.
We also find tight bounds for the minimum number of $k$-dimensional affine subspaces needed to cover $\Lambda \cap K$.

We use these new results to improve the best known lower bound for the maximum number of point-hyperplane incidences by Brass and Knauer~\cite{braKna03}.
For $d \geq 3$ and $\varepsilon \in (0,1)$, we show that there is an integer $r=r(d,\varepsilon)$ such that for all positive integers $n,m$ the following statement is true.
There is a set of $n$ points in $\mathbb{R}^d$ and an arrangement  of $m$ hyperplanes in $\mathbb{R}^d$ with no $K_{r,r}$ in their incidence graph and with at least
$\Omega\left((mn)^{1-(2d+3)/((d+2)(d+3)) - \varepsilon}\right)$ incidences if $d$ is odd and $\Omega\left((mn)^{1-(2d^2+d-2)/((d+2)(d^2+2d-2)) -\varepsilon}\right)$ incidences if $d$ is even.
\end{abstract}

\section{Introduction}

In this paper, we study the minimum number of linear or affine subspaces needed to cover points that are contained in the intersection of a given lattice with a given 0-symmetric convex body.
We also present an application of our results to the problem of estimating the maximum number of incidences between a set of points and an arrangement of hyperplanes.
Consequently, this establishes a new lower bound for the time complexity of so-called partitioning algorithms for Hopcroft's problem.
Before describing our results in more detail, we first give some preliminaries and introduce necessary definitions.

\subsection{Preliminaries}

For linearly independent vectors $b_1,\dots,b_d \in \mathbb{R}^d$, the $d$-dimensional \emph{lattice} $\Lambda = \Lambda(b_1,\dots,b_d)$ with \emph{basis} $\{b_1,\dots,b_d\}$ is the set of all linear combinations  of the vectors $b_1,\dots,b_d$ with integer coefficients.
We define the \emph{determinant of $\Lambda$} as $\det(\Lambda) \colonequals |\det(B)|$, where $B$ is the $d \times d$ matrix with the vectors $b_1,\dots,b_d$ as columns. 
For a positive integer $d$, we use $\mathcal{L}^d$ to denote the set of $d$-dimensional lattices $\Lambda$,
that is, lattices with $\det(\Lambda) \neq 0$.

A convex body $K$ is \emph{symmetric about the origin} $0$ if $K=-K$.
We let  $\mathcal{K}^d$ be the set of $d$-dimensional compact convex bodies in $\mathbb{R}^d$ that are symmetric about the origin.

For a positive integer $n$, we use the abbreviation $[n]$ to denote the set $\{1,2,\dots,n\}$.
A point $x$ of a lattice is called \emph{primitive} if whenever its multiple $\lambda \cdot x$ is a lattice point, then $\lambda$ is an integer.
For $K \in \mathcal{K}_d$, let $\vol(K)$ be the $d$-dimensional Lebesgue measure of~$K$. 
We say that $\vol(K)$ is the \emph{volume of $K$}.
The closed $d$-dimensional ball with the radius $r \in \mathbb{R}$, $r \geq 0$, centered in the origin is denoted by $B^d(r)$.
If $r=1$, we simply write $B^d$ instead of $B^d(1)$.
For $x \in \mathbb{R}^d$, we use $\|x\|$ to denote the Euclidean norm of $x$.

Let $X$ be a subset of $\mathbb{R}^d$.
We use $\aff(X)$ and $\lin(X)$ to denote the \emph{affine hull of $X$} and the \emph{linear hull of~$X$}, respectively.
The dimension of the affine hull of $X$ is denoted by $\dim(X)$.

For functions $f,g \colon \mathbb{N} \to \mathbb{N}$, we write $f(n) \leq O(g(n))$ if there is a fixed constant $c_1$ such that $f(n) \leq c_1 \cdot g(n)$ for all $n \in \mathbb{N}$.
We write $f(n) \geq \Omega(g(n))$ if there is a fixed constant $c_2 > 0$ such that $f(n) \geq c_2 \cdot g(n)$ for all $n \in \mathbb{N}$.
If the constants $c_1$ and $c_2$ depend on some parameters $a_1,\dots,a_t$, then we emphasize this by writing $f(n) \leq O_{a_1,\dots,a_t}(g(n))$ and $f(n) \geq \Omega_{a_1,\dots,a_t}(g(n))$, respectively.
If $f(n) \leq O_{a_1,\dots,a_t}(n)$ and $f(n) \geq \Omega_{a_1,\dots,a_t}(n)$, then we write $f(n) = \Theta_{a_1,\dots,a_t}(n)$.

\subsection{Covering lattice points by subspaces}

We say that a collection $\mathcal{S}$ of subsets in $\mathbb{R}^d$ \emph{covers} a set of points $P$ from $\mathbb{R}^d$ if every point from $P$ lies in some set from $\mathcal{S}$.

Let $d$, $k$, $n$, and $r$ be positive integers that satisfy $1 \leq k \leq d-1$.
We let $a(d,k,n,r)$ be the maximum size of a set $S\subseteq \mathbb{Z}^d \cap B^d(n)$ such that every $k$-dimensional \emph{affine} subspace of $\mathbb{R}^d$ contains at most $r-1$ points of $S$.
Similarly, we let $l(d,k,n,r)$ be the maximum size of a set $S\subseteq \mathbb{Z}^d \cap B^d(n)$ such that every $k$-dimensional \emph{linear} subspace of $\mathbb{R}^d$ contains at most $r-1$ points of $S$. 
We also let $g(d,k,n)$ be the minimum number of $k$-dimensional linear subspaces of $\mathbb{R}^d$ necessary to cover $\mathbb{Z}^d \cap B^d(n)$.

In this paper, we study the functions $a(d,k,n,r)$, $l(d,k,n,r)$, and $g(d,k,n)$ and their generalizations to arbitrary lattices from $\mathcal{L}^d$ and bodies from $\mathcal{K}^d$.
We mostly deal with the last two functions, that is, with covering lattice points by linear subspaces.
In particular, we obtain new upper bounds on $g(d,k,n)$ (Theorem~\ref{thm:upperGeneral}), lower bounds on $l(d,k,n,r)$ (Theorem~\ref{thm:lower}), and we use the estimates for $a(d,k,n,r)$ and $l(d,k,n,r)$ to obtain improved lower bounds for the maximum number of point-hyperplane incidences (Theorem~\ref{thm:incidence}).
Before doing so, we first give a summary of known results, since many of them are used later in the paper.

The problem of determining $a(d,k,n,r)$ is essentially solved.
In general, the set $\mathbb{Z}^d \cap B^d(n)$ can be covered by $(2n+1)^{d-k}$ affine $k$-dimensional subspaces and thus we have an upper bound $a(d,k,n,r) \leq (r-1)(2n+1)^{d-k}$.
This trivial upper bound is asymptotically almost tight for all fixed $d$, $k$, and some $r$, as Brass and Knauer~\cite{braKna03} showed with a probabilistic argument that for every $\varepsilon>0$ there is an $r=r(d,\varepsilon,k) \in \mathbb{N}$ such that for each positive integer $n$ we have 
\begin{equation}
\label{eq:affine}
a(d,k,n,r) \geq \Omega_{d,\varepsilon,k}\left(n^{d-k-\varepsilon}\right).
\end{equation}
For fixed $d$ and $r$, the upper bound is known to be asymptotically tight in the cases $k=1$ and $k=d-1$.
This is shown by considering points on the modular moment surface for $k=1$ and the modular moment curve for $k=d-1$; see~\cite{braKna03}.

Covering lattice points by linear subspaces seems to be more difficult than covering by affine subspaces.
From the definitions we immediately get $l(d,k,n,r) \leq (r-1)g(d,k,n)$.
In the case $k=d-1$ and $d$ fixed, B\'{a}r\'{a}ny, Harcos, Pach, and Tardos~\cite{bhpt01} obtained the following asymptotically tight estimates for the functions $l(d,d-1,n,d)$ and $g(d,d-1,n)$:
\[l(d,d-1,n,d)  = \Theta_d(n^{d/(d-1)})
\hskip 0.5cm \text{ and } \hskip 0.5cm
g(d,d-1,n) = \Theta_d(n^{d/(d-1)}).\]

In fact, B\'{a}r\'{a}ny et al.~\cite{bhpt01} proved stronger results that estimate the minimum number of $(d-1)$-dimensional linear subspaces necessary to cover the set $\Lambda \cap K$ in terms of so-called successive minima of a given lattice $\Lambda \in \mathcal{L}^d$ and a body $K \in \mathcal{K}^d$.

For a lattice $\Lambda \in \mathcal{L}^d$, a body $K \in \mathcal{K}^d$, and $i \in [d]$, we let $\lambda_i(\Lambda, K)$ be the \emph{$i$th successive minimum of~$\Lambda$ and $K$}.
That is, $\lambda_i(\Lambda,K) \colonequals \inf\{\lambda \in\mathbb{R} \colon \dim(\Lambda \cap (\lambda \cdot K))\geq i \}$.
Since $K$ is compact, it is easy to see that the successive minima are achieved.
That is, there are linearly independent vectors $v_1,\dots,v_d$ from~$\Lambda$ such that $v_i \in \lambda_i(\Lambda,K)\cdot K$ for every $i \in [d]$.
Also note that we have $\lambda_1(\Lambda,K) \leq \dots \leq \lambda_d(\Lambda,K)$ and $\lambda_1(\mathbb{Z}^d,B^d(n)) = \cdots = \lambda_d(\mathbb{Z}^d,B^d(n))=1/n$.

\begin{theorem}[\cite{bhpt01}]
\label{thm:Barany}
For an integer $d \geq 2$, a lattice $\Lambda \in \mathcal{L}^d$, and a body $K \in \mathcal{K}^d$, we let $\lambda_i \colonequals \lambda_i(\Lambda,K)$ for every $i \in [d]$.
If $\lambda_d \leq 1$, then the set $\Lambda \cap K$ can be covered with at most 
\[c2^dd^2\log_2{d} \min_{1 \leq j \leq d-1}(\lambda_j\cdots\lambda_d)^{-1/(d-j)}\]
$(d-1)$-dimensional linear subspaces of $\mathbb{R}^d$, where $c$ is some absolute constant.

On the other hand, if $\lambda_d \leq 1$, then there is a subset $S$ of $\Lambda \cap K$ of size 
\[\frac{1-\lambda_d}{16d^2} \min_{1 \leq j \leq d-1}(\lambda_j\cdots\lambda_d)^{-1/(d-j)}\]
such that no $(d-1)$-dimensional linear subspace of $\mathbb{R}^d$ contains $d$ points from $S$.
\end{theorem}

We note that the assumption $\lambda_d \leq 1$ is necessary; see the discussion in~\cite{bhpt01}.
Not much is known for linear subspaces of lower dimension.
We trivially have $l(d,k,n,r) \geq a(d,k,n,r)$ for all $d,k,n,r$ with $1 \leq k \leq d-1$.
Thus $l(d,k,n,r) \geq \Omega_{d,\varepsilon,k}(n^{d-k-\varepsilon})$ for some $r=r(d,\varepsilon,k)$ by~\eqref{eq:affine}.
Brass and Knauer~\cite{braKna03} conjectured that $l(d,k,n,k+1)=\Theta_{d,k}(n^{d(d-k)/(d-1)})$ for $d$ fixed.
This conjecture was refuted by Lefmann~\cite{lefm12} who showed that, for all $d$ and $k$ with $1\leq k \leq d-1$, there is an absolute constant $c$ such that we have $l(d,k,n,k+1) \leq c \cdot n^{d/\lceil k/2 \rceil}$ for every positive integer $n$.
This bound is asymptotically smaller in $n$ than the growth rate conjectured by Brass and Knauer for sufficiently large $d$ and almost all values of $k$ with $1 \leq k \leq d-1$.

Covering lattice points by linear subspaces is also mentioned in the book by Brass, Moser, and Pach~\cite{braMoPa05}, where the authors pose the following problem.

\begin{problem}[{\cite[Problem~6 in Chapter~10.2]{braMoPa05}}]
\label{prob:linearCovering}
What is the minimum number of $k$-dimensional linear subspaces necessary to cover the $d$-dimensional $n \times \cdots \times n$ lattice cube?
\end{problem}

\subsection{Point-hyperplane incidences}

As we will see later, the problem of determining $a(d,k,n,r)$ and $l(d,n,k,r)$ is related to a problem of bounding the maximum number of point-hyperplane incidences.
For an integer $d \geq 2$, let $P$ be a set of $n$ points in $\mathbb{R}^d$ and let $\mathcal{H}$ be an arrangement of $m$ hyperplanes in $\mathbb{R}^d$.
An \emph{incidence between $P$ and $\mathcal{H}$} is a pair $(p,H)$ such that $p \in P$, $H \in \mathcal{H}$, and $p \in H$.
The number of incidences between $P$ and $\mathcal{H}$ is denoted by $\inc(P,\mathcal{H})$.

We are interested in the maximum number of incidences between $P$ and $\mathcal{H}$.
In the plane, the famous \emph{Szemer\'{e}di--Trotter theorem}~\cite{szeTro83} says that the maximum number of incidences between a set of $n$ points in~$\mathbb{R}^2$ and an arrangement of $m$ lines in~$\mathbb{R}^2$ is at most $O((mn)^{2/3}+m+n)$.
This is known to be asymptotically tight, as a matching lower bound was found earlier by Erd\H{o}s~\cite{erd46}.
The current best known bounds are $\approx 1.27 (mn)^{2/3}+m+n$~\cite{pachToth97}\footnote
{
The lower bound claimed by Pach and T\'{o}th~\cite[Remark 4.2]{pachToth97} contains the multiplicative constant $\approx 0.42$.
This is due to a miscalculation in the last equation in the calculation of the number of incidences. The correct calculation is
$I \approx \cdots = 4n \sum_{r=1}^{1/\varepsilon} \phi(r) - 2n\varepsilon^2 \sum_{r=1}^{1/\varepsilon} r^2 \phi(r) \approx 
4n \cdot 3(1/\varepsilon)^2/\pi^2 - 2n\varepsilon^2 (3/2) (1/\varepsilon)^4/\pi^2 = 9 n/(\varepsilon^2\pi^2)$. This leads to
$c \approx 3 \sqrt[3]{3/(4 \pi^2)} \approx 1.27$.
} 
and $\approx 2.44 (mn)^{2/3}+m+n$~\cite{acke15}.

For $d \geq 3$, it is easy to see that there is a set $P$ of $n$ points in $\mathbb{R}^d$ and an arrangement $\mathcal{H}$ of $m$ hyperplanes in $\mathbb{R}^d$ for which the number of incidences is maximum possible, that is $\inc(P,\mathcal{H})=mn$.
It suffices to consider the case where all points from $P$ lie in an affine subspace that is contained in every hyperplane from $\mathcal{H}$.
In order to avoid this degenerate case, we forbid large complete bipartite graphs in the \emph{incidence graph of $P$ and $\mathcal{H}$}, which is denoted by $G(P,\mathcal{H})$.
This is the bipartite graph on the vertex set $P \cup \mathcal{H}$ and with edges $\{p,H\}$ where $(p,H)$ is an incidence between $P$ and $\mathcal{H}$.

With this restriction, bounding $\inc(P,\mathcal{H})$ becomes more difficult and no tight bounds are known for $d \geq 3$.
It follows from the works of Chazelle~\cite{chaze93}, Brass and Knauer~\cite{braKna03}, and Apfelbaum and Sharir~\cite{apfSha07} that the number of incidences between any set $P$ of $n$ points in $\mathbb{R}^d$ and any arrangement $\mathcal{H}$ of $m$ hyperplanes in~$\mathbb{R}^d$ with $K_{r,r}\not\subseteq G(P,\mathcal{H})$ satisfies
\begin{equation}
\label{eq:incidenceUpper}
\inc(P,\mathcal{H}) \leq O_{d,r}\left((mn)^{1-1/(d+1)}+m+n\right).
\end{equation}

We note that an upper bound similar to~\eqref{eq:incidenceUpper} holds in a much more general setting; see the remark in the proof of Theorem~\ref{thm:incidence}.
The best general lower bound for $\inc(P,\mathcal{H})$ is due to a construction of Brass and Knauer~\cite{braKna03}, which gives the following estimate.

\begin{theorem}[\cite{braKna03}]
\label{thm:BrassKnauer}
Let $d \geq 3$ be an integer.
Then for every $\varepsilon > 0$ there is a positive integer $r=r(d,\varepsilon)$ such that for all positive integers $n$ and $m$ there is a set $P$ of $n$ points in $\mathbb{R}^d$ and an arrangement $\mathcal{H}$ of $m$ hyperplanes in~$\mathbb{R}^d$ such that $K_{r,r} \not\subseteq G(P,\mathcal{H})$ and
\[\inc(P,\mathcal{H}) \geq
\begin{cases} 
\Omega_{d,\varepsilon}\left((mn)^{1-2/(d+3) - \varepsilon}\right) &\mbox{if } d \mbox{ is odd and } d > 3 \mbox{,} \\ 
\Omega_{d,\varepsilon}\left((mn)^{1-2(d+1)/(d+2)^2 -\varepsilon}\right) & \mbox{if } d \mbox{ is even,} \\
\Omega_{d,\varepsilon}\left((mn)^{7/10}\right) & \mbox{if } d=3.
\end{cases}\]
\end{theorem}

For $d \geq 4$, this lower bound has been recently improved by Sheffer~\cite{shef15} in a certain non-diagonal case.
Sheffer constructed a set $P$ of $n$ points in $\mathbb{R}^d$, $d \geq 4$, and an arrangement $\mathcal{H}$ of $m=\Theta(n^{(3-3\varepsilon)/(d+1)})$ hyperplanes in $\mathbb{R}^d$ such that $K_{(d-1)/\varepsilon,2} \not \subseteq G(P,\mathcal{H})$ and $\inc(P,\mathcal{H}) \geq \Omega\left((mn)^{1-2/(d+4)-\varepsilon}\right)$.

\section{Our results}

In this paper, we nearly settle Problem~\ref{prob:linearCovering} by proving almost tight bounds for the function $g(d,k,n)$ for a fixed $d$ and an arbitrary $k$ from $[d-1]$.
For a fixed~$d$, an arbitrary $k \in [d-1]$, and some fixed $r$, we also provide bounds on the function $l(d,k,n,r)$ that are very close to the bound conjectured by Brass and Knauer~\cite{braKna03}.
Thus it seems that the conjectured growth rate of $l(d,k,n,r)$ is true if we allow $r$ to be (significantly) larger than $k+1$.

We study these problems in a more general setting where we are given an arbitrary lattice $\Lambda$ from $\mathcal{L}^d$ and a body $K$ from $\mathcal{K}^d$.
Similarly to Theorem~\ref{thm:Barany} by B\'{a}r\'{a}ny et al.~\cite{bhpt01}, our bounds are expressed in terms of the successive minima $\lambda_i(\Lambda,K)$, $i \in [d]$.

\subsection{Covering lattice points by linear subspaces}
First, we prove a new upper bound on the minimum number of $k$-dimensional linear subspaces that are necessary to cover points in the intersection of a given lattice with a body from $\mathcal{K}^d$.

\begin{theorem}
\label{thm:upperGeneral}
For integers $d$ and $k$ with $1\leq k\leq d-1$, a lattice $\Lambda \in \mathcal{L}^d$, and a body $K \in \mathcal{K}^d$, we let $\lambda_i \colonequals \lambda_i(\Lambda,K)$ for $i=1,\dots,d$.
If $\lambda_d \leq 1$, then we can cover $\Lambda \cap K$ with $O_{d,k}(\alpha^{d-k})$ $k$-dimensional linear subspaces of $\mathbb{R}^d$, where 
\[\alpha \colonequals \min_{1 \le j \le k} (\lambda_j\cdots \lambda_d)^{-1/(d-j)}.\]
\end{theorem}

We also prove the following lower bound.

\begin{theorem}
\label{thm:lower}
For integers $d$ and $k$ with $1 \leq k \leq d-1$, a lattice $\Lambda \in \mathcal{L}^d$, and a body $K \in \mathcal{K}^d$, we let $\lambda_i\colonequals \lambda_i(\Lambda,K)$ for $i=1,\dots,d$.
If $\lambda_d \leq 1$, then, for every $\varepsilon \in (0,1)$, there is a positive integer $r=r(d,\varepsilon,k)$ and a set $S \subseteq \Lambda \cap K$ of size at least $\Omega_{d,\varepsilon,k}(((1-\lambda_d)\beta)^{d-k-\varepsilon})$, where
\[\beta \colonequals  \min_{1 \leq j \leq d-1}(\lambda_j\cdots\lambda_d)^{-1/(d-j)},\]
such that every $k$-dimensional linear subspace of $\mathbb{R}^d$ contains at most $r-1$ points from $S$.
\end{theorem}

We remark that we can get rid of the $\varepsilon$ in the exponent if $k=1$ or $k=d-1$; for details, see Theorem~\ref{thm:Barany} for the case $k=d-1$ and the proof in Section~\ref{sec:thmLower} for the case $k=1$.
Also note that in the definition of $\alpha$ in Theorem~\ref{thm:upperGeneral} the minimum is taken over the set $\{1,\dots,k\}$, while in the definition of $\beta$ in Theorem~\ref{thm:lower} the minimum is taken over $\{1,\dots,d-1\}$.
There are examples that show that $\alpha$ cannot be replaced by $\beta$ in Theorem~\ref{thm:upperGeneral}.
It suffices to consider $d=3$, $k=1$, and let $\Lambda$  be the lattice $\{(x_1/n,x_2/2,x_3/2) \in \mathbb{R}^3 \colon x_1,x_2,x_3 \in \mathbb{Z}\}$ for some large positive integer $n$.
Then $\lambda_1(\Lambda,B^3)=1/n$, $\lambda_2(\Lambda,B^3)=1/2$, $\lambda_3(\Lambda,B^3)=1/2$, and thus $\beta = (\lambda_2\lambda_3)^{-1} = 4$.
However, it is not difficult to see that we need at least $\Omega(n)$ 1-dimensional linear subspaces to cover $\Lambda \cap B^3$, which is asymptotically larger than $\beta^2=O(1)$.
On the other hand, $\alpha = (\lambda_1\lambda_2\lambda_3)^{-1/2}$ and $O(\alpha^2)=O(n)$ 1-dimensional linear subspaces suffice to cover $\Lambda \cap B^3$.
We thus suspect that the lower bound can be improved.

Since $\lambda_i(\mathbb{Z}^d,B^d(n))=1/n$ for every $i \in [d]$, we can apply Theorem~\ref{thm:lower} with $\Lambda= \mathbb{Z}^d$ and $K=B^d(n)$ and obtain the following lower bound on $l(d,k,n,r)$.

\begin{corollary}
\label{cor:lower}
Let $d$ and $k$ be integers with $1 \leq k \leq d-1$.
Then, for every $\varepsilon \in (0,1)$, there is an $r=r(d,\varepsilon,k) \in \mathbb{N}$ such that for every $n\in\mathbb{N}$ we have 
\[l(d,k,n,r) \geq \Omega_{d,\varepsilon,k}(n^{d(d-k)/(d-1)-\varepsilon}).\]
\end{corollary}

The existence of the set $S$ from Theorem~\ref{thm:lower} is shown by a probabilistic argument.
It would be interesting to find, at least for some value $1 < k <d-1$, some fixed $r \in\mathbb{N}$, and arbitrarily large $n \in \mathbb{N}$, a construction of a subset $R$ of $\mathbb{Z}^d \cap B^d(n)$  of size $\Omega_{d,k}(n^{d(d-k)/(d-1)})$ such that every $k$-dimensional linear subspace contains at most $r-1$ points from $R$.
Such constructions are known for $k=1$ and $k=d-1$; see~\cite{braKna03,roth51}.

Since we have $l(d,k,n,r) \leq (r-1)g(d,k,n)$ for every $r \in \mathbb{N}$, Theorem~\ref{thm:upperGeneral} and Corollary~\ref{cor:lower} give the following almost tight estimates on $g(d,k,n)$.
This nearly settles Problem~\ref{prob:linearCovering}.

\begin{corollary}
\label{cor:lowerCover}
Let $d$, $k$, and $n$ be integers with $1 \leq k \leq d-1$.
Then, for every $\varepsilon \in (0,1)$, we have 
\[\Omega_{d,\varepsilon,k}(n^{d(d-k)/(d-1)-\varepsilon}) \leq g(d,k,n) \leq O_{d,k}(n^{d(d-k)/(d-1)}).\]
\end{corollary}

\subsection{Covering lattice points by affine subspaces}

For \emph{affine} subspaces, Brass and Knauer~\cite{braKna03} considered only the case of covering the $d$-dimensional $n \times \cdots \times n$ lattice cube by $k$-dimensional affine subspaces.
To our knowledge, the case for general $\Lambda \in \mathcal{L}^d$ and $K \in \mathcal{K}^d$ was not considered in the literature.
We extend the results of Brass and Knauer to covering $\Lambda \cap K$.

\begin{theorem}
\label{thm:coveringAffine}
For integers $d$ and $k$ with $1 \leq k \leq d-1$, a lattice $\Lambda \in \mathcal{L}^d$, and a body $K \in \mathcal{K}^d$, we let $\lambda_i \colonequals \lambda_i(\Lambda,K)$ for $i=1,\ldots,d$.
If $\lambda_d \leq 1$, then the set $\Lambda \cap K$  can be covered with $O_{d,k}((\lambda_{k+1} \cdots \lambda_d)^{-1})$ $k$-dimensional affine subspaces of $\mathbb{R}^d$.

On the other hand, at least $\Omega_{d,k}((\lambda_{k+1}\cdots\lambda_d)^{-1})$ $k$-dimensional affine subspaces of $\mathbb{R}^d$ are necessary to cover $\Lambda \cap K$.
\end{theorem}

\subsection{Point-hyperplane incidences}

As an application of Corollary~\ref{cor:lower}, we improve the best known lower bounds on the maximum number of point-hyperplane incidences in $\mathbb{R}^d$ for $d \geq 4$.
That is, we improve the bounds from Theorem~\ref{thm:BrassKnauer}.
To our knowledge, this is the first improvement on the estimates for $\inc(P,\mathcal{H})$ in the general case during the last 13 years.

\begin{theorem}
\label{thm:incidence}
For every integer $d \geq 2$ and $\varepsilon \in (0,1)$, there is an $r=r(d,\varepsilon) \in \mathbb{N}$ such that for all positive integers $n$ and $m$ the following statement is true.
There is a set $P$ of $n$ points in $\mathbb{R}^d$ and an arrangement $\mathcal{H}$ of $m$ hyperplanes  in $\mathbb{R}^d$ such that $K_{r,r} \not\subseteq G(P,\mathcal{H})$ and 
\[\inc(P,\mathcal{H}) \geq
\begin{cases} 
\Omega_{d,\varepsilon}\left((mn)^{1-(2d+3)/((d+2)(d+3)) - \varepsilon}\right) &\mbox{if } d \mbox{ is odd,} \\ 
\Omega_{d,\varepsilon}\left((mn)^{1-(2d^2+d-2)/((d+2)(d^2+2d-2)) -\varepsilon}\right) & \mbox{if } d \mbox{ is even.} 
\end{cases}\]
\end{theorem}

We can get rid of the $\varepsilon$ in the exponent for $d \leq 3$.
That is, we have the bounds $\Omega((mn)^{2/3})$ for $d=2$ and $\Omega((mn)^{7/10})$ for $d=3$.
For $d=3$, our bound is the same as the bound from Theorem~\ref{thm:BrassKnauer}.
For larger $d$, our bounds become stronger.
In particular, the exponents in the lower bounds from Theorem~\ref{thm:incidence} exceed the exponents from Theorem~\ref{thm:BrassKnauer} by $1/((d+2) (d+3))$ for $d>3$ odd and by $d^2/((d+2)^2 (d^2+2 d-2))$ for $d$ even.
However, the bounds are not tight.
The exponents in the known bounds for $\inc(P,\mathcal{H})$ for small values of $d$ are summarized in Table~\ref{tab:incidence}.

\begin{table}[ht]
\renewcommand{\arraystretch}{1.4}
\begin{center}
\begin{tabular}{c|C{1cm}|C{3.1cm}|C{3.1cm}|C{3.1cm}}
  $d$ & $3$ & $4$ & $5$ & $6$ \\
\hline
 Upper bounds~\cite{apfSha07,braKna03,chaze93,szeTro83} & $3/4$ & $4/5 = 0.8$ & $5/6 \sim 0.833$ & $6/7 \sim 0.857$\\
 Lower bounds from Theorem~\ref{thm:BrassKnauer} & $7/10$ & $13/18-\varepsilon \sim 0.722 - \varepsilon$ & $3/4 - \varepsilon = 0.75 - \varepsilon$ & $25/32 - \varepsilon \sim 0.781 - \varepsilon$\\
 Lower bounds from Theorem~\ref{thm:incidence} & $7/10$ & $49/66-\varepsilon \sim 0.742-\varepsilon$ & $43/56-\varepsilon \sim 0.768 - \varepsilon$ & $73/92 - \varepsilon \sim 0.793 - \varepsilon$\\
\end{tabular}
\end{center}
\caption{Improvements on the exponents in the bounds for the maximum number of point-hyperplane incidences.}
\label{tab:incidence}
\end{table}

In the non-diagonal case, when one of $n$ and $m$ is significantly larger that the other, the proof of Theorem~\ref{thm:incidence} yields the following stronger bound.

\begin{theorem}
\label{thm:incidenceNondiagonal}
For all integers $d$ and $k$ with $0 \leq k \leq d-2$ and for $\varepsilon \in (0,1)$, there is an $r=r(d,\varepsilon,k) \in \mathbb{N}$ such that for all positive integers $n$ and $m$ the following statement is true.
There is a set $P$ of $n$ points in $\mathbb{R}^d$ and an arrangement $\mathcal{H}$ of $m$ hyperplanes  in $\mathbb{R}^d$ such that $K_{r,r} \not\subseteq G(P,\mathcal{H})$ and 
\[\inc(P,\mathcal{H}) \geq \Omega_{d,\varepsilon,k}\left(n^{1-(k+1)/((k+2-1/d)(d-k))-\varepsilon}m^{1-(d-1)/(dk+2d-1)-\varepsilon}\right).\]
\end{theorem}

For example, in the case $m=\Theta(n^{(3-3\varepsilon)/(d+1)})$ considered by Sheffer~\cite{shef15}, Theorem~\ref{thm:incidenceNondiagonal} gives a slightly better bound than $I(P,\mathcal{H}) \geq \Omega((mn)^{1-2/(d+4)-\varepsilon}))$ if we set, for example, $k=\lfloor (d-1)/4 \rfloor$.
However, the forbidden complete bipartite subgraph in the incidence graph is larger than $K_{(d-1)/\varepsilon,2}$.

The following problem is known as the counting version of \emph{Hopcroft's problem}~\cite{braKna03,erickson96}:
given $n$ points in~$\mathbb{R}^d$ and $m$ hyperplanes in $\mathbb{R}^d$, how fast can we count the incidences between them?
We note that  the lower bounds from Theorem~\ref{thm:incidence} also establish the best known lower bounds for the time complexity of so-called \emph{partitioning algorithms}~\cite{erickson96} for the counting version of Hopcroft's problem; see~\cite{braKna03} for more details.

In the proofs of our results, we make no serious effort to optimize the constants.
We also omit floor and ceiling signs whenever they are not crucial.

\section{Proof of Theorem~\ref{thm:upperGeneral}}

Here we show the upper bound on the minimum number of $k$-dimensional linear subspaces needed to cover points from a given $d$-dimensional lattice that are contained in a body $K$ from $\mathcal{K}^d$.
We first prove Theorem~\ref{thm:upperGeneral} in the special case $K=B^d$ (Theorem~\ref{thm:upper}) and then we extend the result to arbitrary $K\in\mathcal{K}^d$.

\subsection{Proof for balls}

Before proceeding with the proof of Theorem~\ref{thm:upperGeneral}, we first introduce some auxiliary results that are used later.
The following classical result is due to Minkowski~\cite{mink10} and shows a relation between $\vol(K)$, $\det(\Lambda)$, and the successive minima of $\Lambda \in \mathcal{L}^d$ and $K \in \mathcal{K}^d$.

\begin{theorem}[Minkowski's second theorem~\cite{mink10}]
\label{thm:2ndMinkowski}
Let $d$ be a positive integer.
For every $\Lambda \in \mathcal{L}^d$ and every $K \in \mathcal{K}^d$, we have
\[\frac{1}{2^d}\cdot\frac{\vol(K)}{\det(\Lambda)} \leq \frac{1}{\lambda_1(\Lambda,K)\cdots\lambda_d(\Lambda,K)} \leq \frac{d!}{2^d} \cdot \frac{\vol(K)}{\det(\Lambda)}.\]
\end{theorem}

A result similar to the first bound from Theorem~\ref{thm:2ndMinkowski} can be obtained if the volume is replaced by the point enumerator; see Henk~\cite{henk02}.

\begin{theorem}[{\cite[Theorem~1.5]{henk02}}]
\label{thm:MinkowskiEnumerator}
Let $d$ be a positive integer.
For every $\Lambda \in \mathcal{L}^d$ and every $K \in \mathcal{K}^d$, we have
\[|\Lambda \cap K| \leq 2^{d-1}\prod_{i=1}^{d}\left\lfloor \frac{2}{\lambda_i(\Lambda,K)}+1\right\rfloor.\]
\end{theorem}

For $\Lambda \in \mathcal{L}^d$ and $K \in \mathcal{K}^d$, let $v_1,\dots,v_d$ be linearly independent vectors such that $v_i \in \Lambda \cap (\lambda_i(\Lambda,K) \cdot K)$ for every $i \in [d]$.
For $d>2$, the vectors $v_1,\dots,v_d$ do not necessarily form a basis of $\Lambda$~{\cite[see Section~X.5]{sieChan89}}.
However, the following theorem shows that there exists a basis with vectors of lengths not much larger than the lengths of $v_1,\dots,v_d$.

\begin{theorem}[First finiteness theorem~{\cite[see Lemma~2 in Section X.6]{sieChan89}}]
\label{thm:firstFniniteness}
Let $d$ be a positive integer.
For every $\Lambda \in \mathcal{L}^d$ and every $K \in \mathcal{K}^d$, there is a basis $\{b_1,\dots,b_d\}$ of $\Lambda$ with $b_i \in (3/2)^{i-1}\lambda_i(\Lambda,K) \cdot K$ for every $i \in [d]$.
\end{theorem}

Now, let $\Lambda$ be a $d$-dimensional lattice with $\lambda_d(\Lambda,B^d)\leq 1$.
Throughout this section, we use $\lambda_i$ to denote the $i$th successive minimum $\lambda_i(\Lambda,B^d)$ for $i=1,\dots,d$.
Let $k$ be an integer with $1 \leq k \leq d-1$.
We show the following result.

\begin{theorem}
\label{thm:upper}
There is a constant $C=C(d,k)$ such that the set $\Lambda \cap B^d$ can be covered with $C \cdot \alpha^{d-k}$ $k$-dimensional linear subspaces of $\mathbb{R}^d$, where 
\[\alpha \colonequals \min_{d-k+1\leq i \leq d}(\lambda_{d-i+1}\cdots\lambda_d)^{-1/(i-1)}.\]
\end{theorem}

This is the same expression as in the statement of Theorem~\ref{thm:upperGeneral}.
We have just chosen a different index notation, since we will work mostly in a dual setting in the proof, where this new expression becomes more natural. 
Let $q$ be an integer from $\{d-k+1,\dots,d\}$ such that $\alpha = (\lambda_{d-q+1}\cdots\lambda_d)^{-1/(q-1)}$, where $\alpha$ is the parameter from the statement of Theorem~\ref{thm:upper}.

In the rest of the section, we prove Theorem~\ref{thm:upper}.
However, since its proof is rather long and complicated, we first give a high-level overview.

We start by proving a weaker upper bound $O_{d,k}((\lambda_k \cdots \lambda_d)^{-1})$ on the number of $k$-dimensional subspaces of $\mathbb{R}^d$ needed to cover $\Lambda \cap B^d$ (Corollary~\ref{cor:MaleQ}).
This bound is obtained from Theorem~\ref{thm:MinkowskiEnumerator} and Lemma~\ref{lem:projection}, which states that, for each $s$ with $0 \leq s \leq d-1$, there is a suitable projection of $\mathbb{R}^d$ on a $(d-s)$-dimensional linear subspace such that the $i$th successive minimum of the image of $\Lambda \cap B^d$ is in $\Theta(\lambda_{i+s})$.
The existence of such projections is proved using Minkowski's second theorem and the First finiteness theorem.
Theorem~\ref{thm:Barany} and the bound from Corollary~\ref{cor:MaleQ} then allows us to to assume $d \geq 4$ and $q  \geq d-k+2$. 
The latter assumption can be used to obtain two estimates on products of successive minima of $\Lambda$ and $B^d$ (Lemma~\ref{lem:lambdaVsAlpha}).

The proof of Theorem~\ref{thm:upper} is then carried out by induction on $d-k$, starting with the case $d-k=1$, in which we cover $\Lambda \cap B^d$ by hyperplanes.
This initial step is treated essentially in the same way as in~\cite{bhpt01} and it is derived using the pigeonhole principle and results of Mahler~\cite{mahl39} and Banaszczyk~\cite{bana93}.
In the resulting covering $\mathcal{S}$ of $\Lambda \cap B^d$ by hyperplanes, the intersection of $\Lambda$ with a hyperplane from $\mathcal{S}$ induces a lattice of lower dimension.
We can thus apply the induction hypothesis on $(\Lambda  \cap H) \cap B^d$ for each hyperplane $H \in \mathcal{S}$. 
Using Minkowski's second theorem and Lemma~\ref{lem:lambdaVsAlpha}, we can show that the larger the norm of the normal vector of $H$ is, the sparser $(\Lambda  \cap H) \cap B^d$ is (Corollary~\ref{cor:indukceNadrovina}). 
Then we partition the hyperplanes from $\mathcal{S}$ according to the lengths of their normal vectors and we sum the sizes of the coverings of $(\Lambda  \cap H) \cap B^d$ by $k$-dimensional subspaces for each $H \in \mathcal{S}$.
Combining Corollary~\ref{cor:indukceNadrovina}, Theorem~\ref{thm:MinkowskiEnumerator}, and the bounds from Lemma~\ref{lem:lambdaVsAlpha}, we finally show that the total sum is bounded from above by $O_{d,k}(\alpha^{d-k})$. 

\vspace{2ex}

Now, as the first step towards the proof of Theorem~\ref{thm:upper}, we prove Corollary~\ref{cor:MaleQ}.
To do so, we prove the following lemma that is also used later in the proof of Theorem~\ref{thm:coveringAffine}.

\begin{lemma}
\label{lem:projection}
Let $d$ and $s$ be integers with $0 \leq s \leq d-1$.
There is a positive integer $r=r(d,s)$ and a projection $p$ of $\mathbb{R}^d$ along $s$ vectors of $\Lambda$ onto a $(d-s)$-dimensional linear subspace $N$ of~$\mathbb{R}^d$ such that $\Lambda \cap B^d$ is mapped to $\Lambda \cap N \cap B^d(r)$ and such that $\lambda_i(\Lambda \cap N, B^d(r) \cap N) = \Theta_{d,s}(\lambda_{i+s})$ for every $i \in [d-s]$.
\end{lemma}
\begin{proof}
If $s=0$, then we set $p$ to be the identity on $\mathbb{R}^d$ and $r \colonequals 1$.
Thus we assume $s \geq 1$.

For $j=0,\ldots,d-1$, we set $r_j \colonequals (2^{d^2}+1)^j$.
For $j=0,\dots,d-1$ and a lattice $\Lambda_j \in \mathcal{L}^{d-j}$, we show that there is a projection $p_j$ of $\mathbb{R}^{d-j}$ along a vector $v_j \in \Lambda_j$ onto a $(d-j-1)$-dimensional linear subspace $N_{j+1}$ of $\mathbb{R}^{d-j}$ such that $\Lambda_j \cap B^{d-j}(r_j)$ is mapped to $\Lambda_j \cap N_{j+1} \cap B^{d-j}(r_{j+1})$ by $p_j$ and such that
\[\lambda_i(\Lambda_j \cap N_{j+1}, B^{d-j}(r_{j+1}) \cap N_{j+1}) \in \Theta_d(\lambda_{i+1}(\Lambda_j, B^{d-j}(r_j)))\]
for every $i \in [d-j-1]$.
We let $\Lambda_0 = \Lambda$ and, for every $j=0, \dots, s-1$, we use the above-defined projection $p_j$ for $\Lambda_j$ and define $\Lambda_{j+1} \colonequals p_j(\Lambda_j) = \Lambda_j \cap N_{j+1}$.
The statement of the lemma is then obtained by setting $p \colonequals p_{s-1} \circ \cdots \circ p_0$.

Let $B=\{b_1,\dots,b_{d-j}\}$ be a basis of $\Lambda_j$ such that $b_i \in (3/2)^{i-1}\lambda_i(\Lambda_j,B^{d-j}(r_j)) \cdot B^{d-j}(r_j)$ for every $i \in [d-j]$.
Such basis exists by the First finiteness theorem (Theorem~\ref{thm:firstFniniteness}).
In particular, 
\begin{equation}
\label{eq:delkab1}
\|b_1\| = \lambda_1(\Lambda_j,B^{d-j}(r_j)) \cdot r_j.
\end{equation}
Let $v_j \colonequals b_1$ and let $N_{j+1}$ be the linear subspace generated by $b_2,\dots,b_{d-j}$.
Let $\Lambda_{j+1}$ be the set $\Lambda_j \cap N_{j+1}$.
Note that $\Lambda_{j+1}$ is a $(d-j-1)$-dimensional lattice with the basis $\{b_2,\dots,b_{d-j}\}$.

We consider the projection $p_j$ onto $N_{j+1}$ along $v_j$.
That is, every $x \in \mathbb{R}^{d-j}$ is mapped to $p_j(x)=\sum_{i=2}^{d-j} t_ib_i$, where $x=\sum_{i=1}^{d-j} t_ib_i$, $t_i \in \mathbb{R}$, is the expression of $x$ with respect to the basis $B$.

We show that $p_j(z) \in \Lambda_{j+1} \cap B^{d-j}(r_{j+1})$ for every $z \in \Lambda_j \cap B^{d-j}(r_j)$.
We have $p_j(z) \in \Lambda_{j+1}$, since $B$ is a basis of $\Lambda_j$ and $B \setminus\{b_1\}$ is a basis of $\Lambda_{j+1}$.
Let $z=\sum_{i=1}^{d-j} t_ib_i$, $t_i \in \mathbb{Z}$, be the expression of $z$ with respect to $B$ and let $v$ be the Euclidean distance between $b_1$ and $N_{j+1}$.

From the definitions of $\Lambda_{j+1}$ and $B$, we have \begin{equation}
\label{eq:odhadLambIndukce}
\lambda_{i+1}(\Lambda_j,B^{d-j}(r_j)) \leq \lambda_i(\Lambda_{j+1},B^{d-j}(r_j)\cap N_{j+1}) \leq \|b_{i+1}\| \leq (3/2)^i\lambda_{i+1}(\Lambda_j,B^{d-j}(r_j))
\end{equation}
for every $i \in [d-j-1]$.
Using Minkowski's second theorem (Theorem~\ref{thm:2ndMinkowski}) twice, the upper bound in~\eqref{eq:odhadLambIndukce}, and the length of $b_1$~\eqref{eq:delkab1}, we obtain 
\begin{align*}
\frac{\vol(B^{d-j}(r_j))}{2^{d-j} \det(\Lambda_j)} &\leq \frac{1}{\lambda_1(\Lambda_j,B^{d-j}(r_j))\cdots \lambda_{d-j}(\Lambda_j,B^{d-j}(r_j))} \qquad \text{by Theorem~\ref{thm:2ndMinkowski} for $\Lambda_j$} \\
&\leq \frac{r_j}{\|b_1\|}\cdot\frac{(3/2)^{(d-j)(d-j-1)/2}}{\lambda_1(\Lambda_{j+1},B^{d-j}(r_j) \cap N_{j+1}) \cdots \lambda_{d-j-1}(\Lambda_{j+1},B^{d-j}(r_j) \cap N_{j+1})} \;\;\; \text{by~\eqref{eq:delkab1}~and~\eqref{eq:odhadLambIndukce}}\\
&\leq \frac{r_j}{\|b_1\|}\cdot\frac{(3/2)^{(d-j)(d-j-1)/2} \cdot (d-j-1)! \cdot \vol(B^{d-j}(r_j) \cap N_{j+1})}{2^{d-j-1} \cdot \det(\Lambda_{j+1})} \; \text{by Theorem~\ref{thm:2ndMinkowski} for $\Lambda_{j+1}$}.
\end{align*}
Since $\det(\Lambda_j) = v \cdot \det(\Lambda_{j+1})$, we can rewrite this expression as 
\[\|b_1\| \leq \frac{r_j\cdot (3/2)^{(d-j)(d-j-1)/2} \cdot (d-j-1)! \cdot 2^{d-j} \cdot \vol(B^{d-j}(r_j) \cap N_{j+1})  \cdot \det(\Lambda_j)}{2^{d-j-1} \cdot \vol(B^{d-j}(r_j)) \cdot \det(\Lambda_{j+1})} \leq 2^{d^2}\cdot v.\]
To derive the last inequality, we use the well-known formula
\[\vol(B^m(r)) = 
\begin{cases} 
\frac{2((m-1)/2)!(4\pi)^{(m-1)/2}}{m!}\cdot r^m &\mbox{if } m \mbox{ is odd,} \\ 
\frac{\pi^{m/2}}{(m/2)!}\cdot r^m & \mbox{if } m \mbox{ is even} 
\end{cases}\]
for the volume of $B^m(r)$, $m,r \in \mathbb{N}$.
Since $\vol(B^{d-j}(r_j) \cap N_{j+1})=\vol(B^{d-j-1}(r_j))$, we have $\vol(B^{d-j}(r_j) \cap N_{j+1})/\vol(B^{d-j}(r_j)) \leq 2^{d-j}/r_j$.

The Euclidean distance between $z$ and $N_{j+1}$ equals $|t_1| \cdot v$, which is at most $r_j$, as $z \in B^{d-j}(r_j)$.
Thus, since $|t_1| \leq r_j/v$ and $1/v \leq 2^{d^2}/ \|b_1\|$, we obtain $|t_1| \leq 2^{d^2} \cdot r_j/\|b_1\|$.
This implies
\[\|p_j(z)\| = \left\|z - t_1b_1 \right\| \leq \|z\| + |t_1| \cdot \|b_1\|  \leq r_j + 2^{d^2}r_j= r_{j+1}\]
and we see that $p_j(z)$ lies in $\Lambda_{j+1} \cap B^{d-j}(r_{j+1})$.

Note that $\lambda_i(\Lambda_{j+1},B^{d-j}(r_{j+1}) \cap N_{j+1}) = (2^{d^2}+1)^{-1} \cdot \lambda_i(\Lambda_{j+1},B^{d-j}(r_j) \cap N_{j+1})$ for every $i\in[d-j-1]$.
Using this fact together with the bounds in~\eqref{eq:odhadLambIndukce}, we obtain
\[\frac{\lambda_{i+1}(\Lambda_j, B^{d-j}(r_j))}{(2^{d^2}+1)} \leq \lambda_i(\Lambda_{j+1}, B^{d-j}(r_{j+1}) \cap N_{j+1}) \leq \frac{(3/2)^{d-j}\lambda_{i+1}(\Lambda_j, B^{d-j}(r_j))}{(2^{d^2}+1)}\]
for every $i \in [d-j-1]$.
That is,
\begin{align*}
\lambda_i(\Lambda_{j+1}, B^{d-j-1}(r_{j+1})) = 
\lambda_i(\Lambda_{j+1}, B^{d-j}(r_{j+1}) \cap N_{j+1}) = 
\Theta_{d}(\lambda_{i+1}(\Lambda_j, B^{d-j}(r_j))).
\end{align*}

Consequently, for $N \colonequals N_s$ and $r \colonequals r_s$, we have $\Lambda_s = \Lambda \cap N$ and
\[
\lambda_i(\Lambda \cap N, B^{d}(r) \cap N) = 
\lambda_i(\Lambda_s, B^{d-s}(r_s)) = 
\Theta_{d,s}(\lambda_{i+s}(\Lambda_0, B^{d}(r_0))) =
\Theta_{d,s}(\lambda_{i+s})
\]
for every $i \in [d-s]$.
\qed
\end{proof}

\begin{corollary}
\label{cor:MaleQ}
The set $\Lambda \cap B^d$ can be covered with $O_{d,k}((\lambda_k\cdots\lambda_d)^{-1})$ $k$-dimensional linear subspaces of~$\mathbb{R}^d$.
\end{corollary}
\begin{proof}
By Lemma~\ref{lem:projection}, there is a positive integer $r=r(d,k-1)$ and a projection $p$ of $\mathbb{R}^d$ along $k-1$ vectors $b_1,\dots,b_{k-1} \in \Lambda$ onto a $(d-k+1)$-dimensional linear subspace $N$ of $\mathbb{R}^d$ such that $\Lambda \cap B^d$ is mapped to $\Lambda \cap N \cap B^d(r)$ and such that $\lambda'_i \colonequals \lambda_i(\Lambda \cap N, B^d(r) \cap N) = \Theta_{d,k}(\lambda_{i+k-1})$ for every $i \in [d-k+1]$.
We use $\Lambda_N$ to denote the $(d-k+1)$-dimensional sublattice $\Lambda \cap N$ of $\Lambda$.

We consider the set $\mathcal{S} \colonequals \{\lin(\{y,b_1,\dots,b_{k-1}\}) \colon y \in (\Lambda_N\setminus\{0\}) \cap B^d(r)\}$.
Then $\mathcal{S}$ consists of $k$-dimensional linear subspaces.
By Theorem~\ref{thm:MinkowskiEnumerator}, the size of $\mathcal{S}$ is at most 
\[
|\Lambda_N \cap B^d(r)| \leq 2^{d-k}\prod_{i=1}^{d-k+1}\left\lfloor \frac{2}{\lambda'_i} +1\right\rfloor 
\leq O_{d,k}\left(\prod_{i=1}^{d-k+1} \frac{1}{\lambda'_i}\right)
\leq O_{d,k}((\lambda_k\cdots\lambda_d)^{-1}),
\]
where the second inequality follows from the assumption $\lambda_d \leq 1$, as then $\lambda'_{d-k+1} \leq O_{d,k}(\lambda_d)$ implies $\lambda'_1 \leq \cdots \leq \lambda'_{d-k+1} \leq O_{d,k}(1)$.
The last inequality is obtained from $\lambda'_i \geq \Omega_{d,k}(\lambda_{i+k-1})$ for every $i \in [d-k+1]$.
Moreover, $\mathcal{S}$ covers $\Lambda \cap B^d$, since for every $z \in \Lambda \cap B^d$, 
$p(z)\in \Lambda_N \cap B^d(r)$, therefore $p(z) \in S$ for some $S \in \mathcal{S}$ and, since
$z \in \lin(p(z), b_1, \ldots, b_{k-1})$, we have $z \in S$.
\qed
\end{proof}

The case $k=1$ of Theorem~\ref{thm:upper} follows from Theorem~\ref{thm:MinkowskiEnumerator} (and also from
Corollary~\ref{cor:MaleQ}).
The case $k=d-1$ was shown by B\'{a}r\'any et al.~\cite{bhpt01}; see Theorem~\ref{thm:Barany}.
Therefore we may assume $d \ge 4$.
Corollary~\ref{cor:MaleQ} also provides the same bound as Theorem~\ref{thm:upper} if $q=d-k+1$, thus  we assume $q \geq d-k+2$ in the rest of the proof.

\begin{lemma}
\label{lem:lambdaVsAlpha}
If $q \geq d-k+2$, then the following two statements are satisfied.
\begin{enumerate}[label=(\roman*)]
\item\label{item-alpha1} We have $1/\lambda_i \leq \alpha$ for every $i \in \{d-q+1,\dots,d\}$,
\item\label{item-alpha3} $(\lambda_{d-i+2}\cdots\lambda_d)^{(q-i+1)/(i-2)} \leq \lambda_{d-q+1} \cdots \lambda_{d-i+1}$ for every $i \in \{3,\dots,d-k+2\}$.
\end{enumerate}
\end{lemma}
\begin{proof}
For part~\ref{item-alpha1}, it suffices to show $1/\lambda_{d-q+1} \leq \alpha$, as $\lambda_{d-q+1} \leq \dots \leq \lambda_d$.
Suppose for contradiction that $1/\lambda_{d-q+1} > \alpha = (\lambda_{d-q+1}\cdots\lambda_d)^{-1/(q-1)}$.
Then we can rewrite this inequality as 
\begin{align*}
\lambda_{d-q+1}^{-1+1/(q-1)} &> (\lambda_{d-q+2}\cdots\lambda_d)^{-1/(q-1)} \\
\lambda_{d-q+1}^{-1/(q-1)} &> (\lambda_{d-q+2}\cdots\lambda_d)^{-1/((q-2)(q-1))} = (\lambda_{d-q+2}\cdots\lambda_d)^{-1/(q-2)+1/(q-1)}.
\end{align*}
The last expression can be further rewritten as 
\[(\lambda_{d-q+1}\cdots\lambda_d)^{-1/(q-1)} > (\lambda_{d-q+2}\cdots\lambda_d)^{-1/(q-2)},\]
and, since the left-hand side equals $\alpha$, this contradicts the choice of $\alpha$.
Here we use the assumption $q \geq d-k+2$, as then $q-1$ lies in the set $\{d-k+1,\dots,d\}$. 

For part~\ref{item-alpha3}, suppose first for contradiction that the inequality is not true for $i=d-k+2$.
That is, $\lambda_{d-q+1}\cdots\lambda_{k-1}<(\lambda_k\cdots \lambda_d)^{(q-d+k-1)/(d-k)}$.
Then we rewrite this expression as 
\[(\lambda_{d-q+1}\cdots\lambda_{k-1})^{1/(q-1)}<(\lambda_k\cdots \lambda_d)^{(q-d+k-1)/((d-k)(q-1))}=(\lambda_k\cdots \lambda_d)^{1/(d-k)-1/(q-1)}\]
 and further as $(\lambda_k\cdots \lambda_d)^{-1/(d-k)}<(\lambda_{d-q+1}\cdots\lambda_d)^{-1/(q-1)}=\alpha$.
However, this is a contradiction with the definition of~$\alpha$. 

Now we show that if the inequality is satisfied for some $i \in \{4,\dots,d-k+2\}$, then it is true also for $i-1$.
Assume that we have $(\lambda_{d-i+2}\cdots\lambda_d)^{(q-i+1)/(i-2)} \leq \lambda_{d-q+1} \cdots \lambda_{d-i+1}$ and suppose for contradiction that $(\lambda_{d-i+3}\cdots\lambda_d)^{(q-i+2)/(i-3)} > \lambda_{d-q+1} \cdots \lambda_{d-i+2}$.
We rewrite the second inequality as $(\lambda_{d-i+3}\cdots\lambda_d) > (\lambda_{d-q+1} \cdots \lambda_{d-i+2})^{(i-3)/(q-i+2)}$.
Then we have
\begin{align*}
(\lambda_{d-i+2}\cdots\lambda_d)^{(q-i+1)/(i-2)} &> \lambda_{d-i+2}^{(q-i+1)/(i-2)}\cdot (\lambda_{d-q+1}\cdots\lambda_{d-i+2})^{(i-3)(q-i+1)/((i-2)(q-i+2))}\\
&=\lambda_{d-i+2}^{(q-1)(q-i+1)/((i-2)(q-i+2))}\cdot (\lambda_{d-q+1}\cdots\lambda_{d-i+1})^{(i-3)(q-i+1)/((i-2)(q-i+2))}.
\end{align*}
Since $\lambda_{d-i+2} \geq \lambda_{d-i+1} \geq \cdots \geq \lambda_{d-q+1}$, we have $\lambda_{d-i+2}^{q-i+1} \geq \lambda_{d-q+1}\cdots\lambda_{d-i+1}$.
Thus we obtain
\begin{align*}
(\lambda_{d-i+2}\cdots\lambda_d)^{(q-i+1)/(i-2)} &>  (\lambda_{d-q+1}\cdots\lambda_{d-i+1})^{(q-1 + (i-3)(q-i+1))/((i-2)(q-i+2))}\\
&= \lambda_{d-q+1}\cdots\lambda_{d-i+1},
\end{align*}
which contradicts our assumption.
\qed
\end{proof}

We use $\Lambda^*$ to denote the \emph{dual lattice of $\Lambda$}.
That is, $\Lambda^*$ is the set of vectors $y$ from $\mathbb{R}^d$ that satisfy $\langle x,y \rangle \in \mathbb{Z}$ for every $x \in \Lambda$. 

In the rest of the section, we use $\mu_i$ to denote $\lambda_i(\Lambda^*,B^d)$ for every $i\in[d]$ and we let
\[\alpha = \min_{d-k+1\leq i \leq d}(\lambda_{d-i+1}\cdots\lambda_d)^{-1/(i-1)}\]
be the parameter from the statement of Theorem~\ref{thm:upper}.
It follows from the results of Mahler~\cite{mahl39} and Banaszczyk~\cite{bana93} that 
\begin{equation}
\label{eq:minimaRevers}
1 \leq \lambda_i \cdot \mu_{d-i+1} \leq d
\end{equation}
holds for every $i\in [d]$.
Observe that $\mu_1 \geq 1$ and $\alpha = \Theta_{d,k}((\mu_1\cdots\mu_q)^{1/(q-1)})$ by~\eqref{eq:minimaRevers} and by the assumption $\lambda_d \leq 1$.
We also recall that $\lambda_1 \leq \dots \leq \lambda_d$ and $\mu_1 \leq \dots \leq \mu_d$.

We now prove Theorem~\ref{thm:upper} by induction on $d-k$.
The case $d-k=1$ is treated similarly as in the proof of Theorem~\ref{thm:Barany} by B\'{a}r\'{a}ny et al.~\cite{bhpt01}.
Let $w_1, \ldots, w_d$ be linearly independent vectors from $\mathbb{R}^d$ such that $w_i \in \Lambda^* \cap \mu_i B^d$ for every $i\in[d]$.
The existence of every $w_i$ is guaranteed from the definition of $\mu_i$.

For a positive real number $\gamma$, we define sets
\[D^+_{\gamma} \colonequals \left\{ \sum_{i=1}^q a_i w_i \colon a_i \in \left[0, \frac{\gamma}{\mu_i}\right] \cap \Zbb \right\} \hskip 0.5cm \text{ and } \hskip 0.5cm
D_{\gamma} \colonequals \left\{ \sum_{i=1}^q a_i w_i \colon a_i \in \left[-\frac{\gamma}{\mu_i}, \frac{\gamma}{\mu_i}\right] \cap \Zbb \right\}.\]

The size of $D^+_\gamma$ is $\prod_{i=1}^q (\lfloor \gamma/\mu_i \rfloor + 1) \geq \prod_{i=1}^q \frac{\gamma}{\mu_i} = \gamma^q/(\mu_1\cdots\mu_q)$.
The inequality~\eqref{eq:minimaRevers} implies $\alpha^{q-1} \leq \mu_1\cdots\mu_q \leq d^q\alpha^{q-1}$.
Thus $|D^+_\gamma| \geq \gamma^q/(d^q \alpha^{q-1})$.
For a sufficiently large constant $c=c(d,k)>0$, the set $D^+ \colonequals D^+_{c\alpha}$ thus satisfies $|D^+| \geq c^q \alpha^q/(d^q\alpha^{q-1}) = c^q\alpha/d^q > 2qc\alpha+1$.
The last inequality follows from our assumption $\lambda_d\leq 1$, as then $\alpha \geq 1$.
We also use the bound $q \geq 2$.
By part~\ref{item-alpha1} of Lemma~\ref{lem:lambdaVsAlpha} and by~\eqref{eq:minimaRevers}, we have $\mu_1 \leq \dots \leq \mu_q \leq d\alpha$.
Thus $\lfloor \gamma / \mu_i\rfloor + 1 \leq 2\gamma/\mu_i$ for every $\gamma \geq d\alpha$ and every $i \in [q]$.
Therefore $|D^+_\gamma| \leq 2^q\gamma^q/(\mu_1\cdots\mu_q) \leq 2^q\gamma^q/\alpha^{q-1}$ and, in particular, $|D^+| \leq 2^qc^q\alpha$. That is, we have
\[
2qc\alpha+1 < |D^+| \le 2^qc^q\alpha.
\]

Let $D \colonequals D_{c\alpha}$.
We show that for every $x \in \Lambda \cap B^d$ there exists $z \in D \setminus \{0\}$ perpendicular to $x$.
Let $x$ be an arbitrary element from $\Lambda \cap B^d$.
For every $y \in D^+$, we have $|\langle x,y\rangle| = |\langle x,\sum_{i=1}^q a_iw_i\rangle| = |\sum_{i=1}^q a_i\langle x,w_i\rangle| \leq \sum_{i=1}^q a_i|\langle x,w_i\rangle|$ for some integers $a_i \in [0,c\alpha/\mu_i]$.
Every $w_i$ is an element of $\mu_i B^d$ and thus the Cauchy--Schwarz inequality implies $|\langle x,w_i\rangle| \leq \mu_i$.
Using $a_i \leq c\alpha/\mu_i$, we thus see that $|\langle x,y\rangle| \leq \sum_{i=1}^q\frac{c\alpha}{\mu_i}\cdot \mu_i = q c \alpha$.
Since $y \in D^+ \subseteq \Lambda^*$, we have $\langle x,y\rangle \in \mathbb{Z}$.
Therefore $\langle x,y\rangle$ attains at most $2qc\alpha+1$ values.
Since $|D^+|>2qc\alpha+1$, the pigeonhole principle implies that there exist distinct $y_1 \in D^+$ and $y_2 \in D^+$ with $\langle x,y_1\rangle=\langle x,y_2\rangle$.
The element $z \colonequals y_1-y_2$ then lies in $D\setminus\{0\}$ and satisfies $\langle x,z\rangle = 0$.

For a vector $z \in \mathbb{R}^d \setminus \{0\}$, we define a hyperplane $H(z) \colonequals \{x \in \mathbb{R}^d \colon \langle x,z\rangle = 0 \}$.
Let $D'$ be the set of primitive points from $D \setminus\{0\}$.
Consider the set $\mathcal{S} \colonequals \{H(z) \colon z \in D'\}$ of hyperplanes in $\mathbb{R}^d$.
Then $\mathcal{S}$ covers $\Lambda \cap B^d$ and contains at most $|D'| < |D| \leq 2^q|D^+| \leq 2^{2q}c^q\alpha=O_d(\alpha^{d-k})$ hyperplanes.
This finishes the base of the induction.

For the inductive step, assume that $d-k \geq 2$.
Consider the set $\mathcal{S}$ of hyperplanes in $\mathbb{R}^d$ that has been constructed in the base of the induction.
For every hyperplane $H \in \mathcal{S}$, let $\Lambda_H$ be the set $\Lambda \cap H$.
Note that $\Lambda_H$ is a lattice of dimension at most $d-1$.
We now proceed inductively and cover each set $\Lambda_H \cap B^d$ using the inductive hypothesis for $\Lambda_H$ and $k$.
Later, we show that the total number of $k$-dimensional subspaces used in the covering of the sets $\Lambda_H\cap B^d$, $H \in \mathcal{S}$, is at most $O_{d,k}(\alpha^{d-k})$.
To do so, we employ the fact that, for every $z \in D'$, the larger $\|z\|$ is, the fewer $k$-dimensional subspaces we need to cover $\Lambda_{H(z)} \cap B^d$.

\begin{lemma}
\label{lem:indukceNadrovina}
Let $z$ be a point from $D'$ and let $\lambda'_i\colonequals\lambda_i(\Lambda_{H(z)},B^d)$ for every $i \in [d-1]$.
If $q \geq d-k+2$ then for every $r \in\{k+1,\dots,d-1\}$, we have
\[
\alpha'_r \colonequals \min_{r-k+1 \leq i \leq r}(\lambda'_{r-i+1}\cdots\lambda' _r)^{-1/(i-1)} \le O_{d,k}\left(\left(\frac{\mu_1 \cdots \mu_q}{\|z\|}\right)^{(d-k-1)/((q-2)(r-k))}\right).
\]
\end{lemma}

Note that $q>2$ according to our assumptions $q \geq d-k+1$ and $k \leq d-2$.

\begin{proof}
The vector $z$ partitions the lattice $\Lambda$ into \emph{layers} $(L_i)_{i\in\mathbb{Z}}$, where $L_i \colonequals \{x \in \Lambda \colon \langle x,z\rangle=i\}$.
Since $z$ is primitive, there is a basis $B$ of $\Lambda^*$ with a column $z$ (see Lemma~1 of Section X.4 in~\cite{sieChan89}).
Then $B'\colonequals (B^{-1})^\top$ is a basis of $\Lambda$ and thus there is a column $v$ of $B'$ with $\langle v,z \rangle =1$.
We have $v \in L_1$ and $i \cdot v \in L_i$ for every $i \in\mathbb{Z}$.
Thus every layer $L_i$ satisfies $L_i = i \cdot v + L_0$ and, in particular, $L_0$ is a $(d-1)$-dimensional sublattice of $\Lambda$.
The Euclidean distance between $\aff(L_i)$ and $\aff(L_{i+1})$ is $1/\|z\|$.
This is because, on one hand, $y\colonequals i \cdot z/\|z\|^2 \in \aff(L_i)$, $y'\colonequals (i+1) \cdot z/\|z\|^2 \in \aff(L_{i+1})$, and $\|y-y'\|=1/\|z\|$.
On the other hand, for all $x \in \aff(L_i)$ and $x'\in \aff(L_{i+1})$, the Cauchy--Schwarz inequality implies
\[\|x-x'\|\|z\| \geq |\langle x-x',z\rangle| = |\langle x,z \rangle - \langle x',z \rangle| = |i-(i+1)|=1\]
and hence $\|x-x'\| \geq 1/\|z\|$.

Since $\Lambda_{H(z)} = \{x \in \Lambda \colon \langle x,z\rangle=0\}$, the lattice $\Lambda_{H(z)}$ is the layer $L_0$ of $\Lambda$.
The affine hull of the closest layer is in the Euclidean distance $1/\|z\|$ from~$\aff(\Lambda_{H(z)})$ and it contains a vector $v$ of $\Lambda$ such that $L_i=i\cdot v +L_0$ for every $i\in \mathbb{Z}^d$.
Thus if $B''$ is a basis of $\Lambda_{H(z)}$, then $B''$ with the column $v$ added is a basis of $\Lambda$.
The parallelotope formed by the vectors of $B''$ and $v$ has volume $\det(\Lambda_{H(z)})/\|z\|$. 
Thus $\det(\Lambda_{H(z)})=\|z\|\det(\Lambda)$. 

Using Minkowski's second theorem (Theorem~\ref{thm:2ndMinkowski}) twice and the fact $\det(\Lambda_{H(z)})=\|z\|\det(\Lambda)$, we have 
\begin{equation}
\label{eq:prepisNormou}
\frac{1}{\lambda'_{1} \cdots \lambda'_{d-1}} 
= \Theta_d\left(\frac{\vol(B^{d-1})}{\det(\Lambda_{H(z)})}\right)
= \Theta_d\left(\frac{1}{\det(\Lambda_{H(z)})}\right)
= \Theta_d\left(\frac{\vol(B^d)}{\det(\Lambda) \|z\|}\right) 
= \Theta_d\left(\frac{1}{\lambda_{1} \cdots \lambda_d \|z\| } \right).
\end{equation}

We now show that 
\begin{equation}
\label{eq:prvniLambdy}
\lambda'_1\cdots\lambda'_{d-q} = \Theta_{d,k}(\lambda_1\cdots\lambda_{d-q}).
\end{equation}
Since $\Lambda_{H(z)} \subseteq \Lambda$, we have $\lambda'_i \geq \lambda_i$ for every $i \in[d-q]$ and thus $\lambda'_1\cdots\lambda'_{d-q} \geq \lambda_1\cdots\lambda_{d-q}$.
For the other inequality, let $w_1,\dots,w_d$ be linearly independent vectors from $\Lambda^*$ such that $\|w_i\| = \mu_i$ for every $i \in [d]$.
The existence of every vector $w_i$ is guaranteed by the definition of $\mu_i$.
Clearly, every $w_i$ is primitive.
Let $L$ be the orthogonal complement of $\lin(\{w_1,\dots,w_q\})$ and let $\Lambda_L$ be the $(d-q)$-dimensional lattice $\Lambda \cap L$.
By iterating the proof of~\eqref{eq:prepisNormou} for the vectors $w_1,\dots,w_q$, we obtain 
\[\prod_{i=1}^{d-q} \frac{1}{\lambda_i(\Lambda_L,B^d)} = \Theta_{d,k}\left(\frac{1}{\lambda_1\cdots\lambda_d \cdot \|w_1\|\cdots\|w_q\|}\right) 
= \Theta_{d,k}\left(\frac{1}{\lambda_1\cdots\lambda_d \cdot \mu_1 \cdots \mu_q}\right) = \Theta_{d,k}\left(\frac{1}{\lambda_1 \cdots \lambda_{d-q}}\right),\]
where the last equality follows from~\eqref{eq:minimaRevers}.
Since $z$ lies in $D'$, we have $z=\sum_{i=1}^q a_iw_i$ for some $a_i \in \mathbb{Z}$ and thus $L \subseteq H(z)$ and $\Lambda_L \subseteq \Lambda_{H(z)}$.
In particular, we have
\[\Omega_{d,k}\left(\frac{1}{\lambda_1 \cdots \lambda_{d-q}}\right) \leq \prod_{i=1}^{d-q} \frac{1}{\lambda_i(\Lambda_L,B^d)} \leq \frac{1}{\lambda'_1 \cdots \lambda'_{d-q}},\]
which proves~\eqref{eq:prvniLambdy}.

By combining the estimates~\eqref{eq:prepisNormou} and~\eqref{eq:prvniLambdy}, we obtain
\begin{equation}
\label{eq:mink}
\frac{1}{\lambda'_{d-q+1} \cdots\lambda'_{d-1}} = \Theta_{d,k}\left(\frac{1}{\lambda_{d-q+1} \cdots \lambda_d \|z\| } \right).
\end{equation}

Since $d-k+1 \leq q$ and $k<r$, we have $d-q+2 \leq r \leq d-1$.
If $r=d-1$, then, using the definition of $\alpha'_r$, $d-k+1\leq q$, \eqref{eq:mink}, and~\eqref{eq:minimaRevers}, we have
\[
\alpha'_r \leq \frac{1}{(\lambda'_{d-q+1} \cdots \lambda'_{d-1})^{1/(q-2)}}
= \Theta_{d,k}\left(\frac{1}{(\lambda_{d-q+1} \cdots \lambda_d \|z\| )^{1/(q-2)}}\right) = \Theta_{d,k}\left(\left(\frac{\mu_1\cdots\mu_q}{\|z\|}\right)^{1/(q-2)}\right),
\]
which settles the claim since $d-k-1 = r-k$.

Assume $r \leq d-2$.
Since $z$ lies in $D'$, we have $z = \sum_{i=1}^qa_iw_i$ for some integers $a_i \in [-c\alpha/\mu_i,c\alpha/\mu_i]$.
Then
\[\|z\|=\left\|\sum_{i=1}^q a_iw_i\right\| \leq \sum_{i=1}^q |a_i| \|w_i\| = \sum_{i=1}^q |a_i|\mu_i \leq \sum_{i=1}^q c\alpha = qc\alpha.\]

That is $\|z\| \leq qc\alpha$, and we have
\begin{equation}
\label{eq:vzdalenostOdhad}
\frac{1}{\lambda_{d-q+1} \cdots \lambda_d \|z\|} \geq \frac{1}{\lambda_{d-q+1}\cdots\lambda_dqc\alpha} =
\frac{1}{qc(\lambda_{d-q+1} \cdots \lambda_d)^{(q-2)/(q-1)}}.
\end{equation}

From~\eqref{eq:mink}  and~\eqref{eq:vzdalenostOdhad}, we obtain
\begin{equation}
\label{eq:odhadKonce}
\frac{1}{\lambda'_{d-q+1} \cdots  \lambda'_{d-1}} 
= \Theta_{d,k}\left(\frac{1}{\lambda_{d-q+1} \cdots \lambda_d \|z\| } \right) 
\geq \Omega_{d,k}\left( \frac{1}{(\lambda_{d-q+1} \cdots \lambda_d)^{(q-2)/(q-1)}}\right).
\end{equation}

From~\eqref{eq:odhadKonce}, we have 
\[\frac{1}{\lambda'_{r+1}\cdots\lambda'_{d-1}} 
=\frac{\lambda'_{d-q+1}\cdots\lambda'_r}{\lambda'_{d-q+1}\cdots\lambda'_{d-1}} 
\geq \Omega_{d,k}\left(\frac{\lambda'_{d-q+1}\cdots\lambda'_r}{(\lambda_{d-q+1}\cdots\lambda_r)^{(q-2)/(q-1)}}\cdot\left(\frac{1}{\lambda_{r+1}\cdots\lambda_d}\right)^{(q-2)/(q-1)}\right).\]
Since $k+1 \leq r \leq d-2$, we have $d-r+1 \in \{3,\dots,d-k+2\}$.
Therefore, using the assumption $q \geq d-k+2$, we may apply part~\ref{item-alpha3} of Lemma~\ref{lem:lambdaVsAlpha} with $i \colonequals d-r+1$ and bound the last expression from below by
\begin{align*}
\Omega_{d,k}&\left(\frac{\lambda'_{d-q+1}\cdots\lambda'_r}{(\lambda_{d-q+1}\cdots\lambda_r)^{(q-2)/(q-1)}}\cdot\left(\frac{1}{\lambda_{d-q+1}\cdots\lambda_r}\right)^{\frac{(q-2)(d-r-1)}{(q-1)(q-d+r)}}\right)
\\&
=\Omega_{d,k}\left(\frac{\lambda'_{d-q+1}\cdots\lambda'_r}{(\lambda_{d-q+1}\cdots\lambda_r)^{(q-2)/(q-d+r)}}\right).
\end{align*}
Since $\Lambda_{H(z)} \subseteq \Lambda$, we have $\lambda'_i \geq \lambda_i$ for every $i \in [d-1]$ and thus we can use the obtained lower bound on $1/(\lambda'_{r+1}\cdots\lambda'_{d-1})$ and derive
\begin{align*}
\frac{1}{\lambda'_{d-q+1}\cdots\lambda'_{d-1}} &\geq \Omega_{d,k}\left(\frac{1}{\lambda'_{d-q+1}\cdots\lambda'_r}\cdot\frac{\lambda'_{d-q+1}\cdots\lambda'_r}{(\lambda_{d-q+1}\cdots\lambda_r)^{(q-2)/(q-d+r)}}\right) \\
&\geq 
\Omega_{d,k}\left(\frac{1}{(\lambda'_{d-q+1}\cdots\lambda'_r)^{(q-2)/(q-d+r)}}\right).\end{align*}
In particular, since $q \geq d-k+1$, the definition of $\alpha'_r$ implies
\begin{align*}
\alpha'_r \leq \frac{1}{(\lambda'_{d-q+1}\cdots\lambda'_r)^{1/(q-d+r-1)}} &\leq O_{d,k}\left(\left(\frac{1}{\lambda'_{d-q+1}\cdots\lambda'_{d-1}}\right)^{(q-d+r)/((q-2)(q-d+r-1))}\right)\\
&\leq O_{d,k}\left(\left(\frac{1}{\lambda_{d-q+1}\cdots\lambda_d\|z\|}\right)^{(q-d+r)/((q-2)(q-d+r-1))}\right),
\end{align*}
where the last inequality follows from~\eqref{eq:mink}.

It remains to show that the exponent in the last term is at most $(d-k-1)/((q-2)(r-k))$, as then the rest follows from~\eqref{eq:minimaRevers}.
Using our assumptions $d-k+1 \leq q$ and $r \leq d-2$, we have
\[\frac{(q-d+r)}{(q-d+r-1)} = 1 + \frac{1}{q-d+r-1} \leq 1 + \frac{d-r-1}{r-k} = \frac{d-k-1}{r-k}.\]
\qed
\end{proof}

\begin{corollary}
\label{cor:indukceNadrovina}
If $z$ is a point from $D'$ and $q \geq d-k+2$, then $\Lambda_{H(z)} \cap B^d$ can be covered with
\[
O_{d,k}\left(\left(\frac{\mu_1 \cdots \mu_q}{\|z\|}\right)^{(d-k-1)/(q-2)}\right)
\]
$k$-dimensional linear subspaces of $\mathbb{R}^d$.
\end{corollary}
\begin{proof}
Following the notation from the statement of Lemma~\ref{lem:indukceNadrovina}, we let $\lambda'_i$ be the $i$th successive minimum $\lambda_i(\Lambda_{H(z)},B^d)$ for every $i \in [d-1]$.
Let $r$ be the largest integer from $[d-1]$ such that $\lambda'_r \leq 1$.
We assume that $r$ exists, as otherwise $\Lambda_{H(z)} \cap B^d = \{0\}$.
From the definition of $\lambda'_r$, we have $\dim(\Lambda_{H(z)} \cap B^d)=r$.
If $r \leq k$, then $\Lambda_{H(z)} \cap B^d$ is contained in a $k$-dimensional linear subspace, which clearly covers $\Lambda_{H(z)} \cap B^d$.
The statement then follows, since $\|z\| \leq qc\alpha$ and $\mu_1\cdots\mu_q = \Theta_{d,k}(\alpha^{q-1})$ imply $(\mu_1 \cdots \mu_q/\|z\|)^{(d-k-1)/(q-2)} \geq \Omega_{d,k}(1)$.
Thus we assume $r>k$.

Let $v'_1,\dots,v'_{d-1}$ be linearly independent vectors such that $v'_i \in \Lambda_{H(z)} \cap (\lambda'_i \cdot B^d)$ for every $i \in [d-1]$.
Let $\Lambda'$ be the $r$-dimensional lattice $\Lambda_{H(z)} \cap \lin(\{v'_1,\dots,v'_r\})$.
Note that $\lambda_i(\Lambda',B^d)=\lambda'_i$ for every $i\in[r]$.
Since $\Lambda' \in\mathcal{L}^r$ and $\lambda_r(\Lambda',B^d)\leq1$, we apply the inductive hypothesis of Theorem~\ref{thm:upper} for $r$ and $k$ and cover $\Lambda' \cap B^d$ with $O_{d,k}((\alpha'_r)^{r-k})$ $k$-dimensional linear subspaces, where $\alpha'_r \colonequals \min_{r-k+1 \leq i \leq r}(\lambda'_{r-i+1}\cdots\lambda' _r)^{-1/(i-1)}$.
By Lemma~\ref{lem:indukceNadrovina}, we have 
\[(\alpha'_r)^{r-k} \le O_{d,k}\left(\left(\frac{\mu_1 \cdots \mu_q}{\|z\|}\right)^{(d-k-1)/(q-2)}\right).\]
The rest follows from $\Lambda_{H(z)} \cap B^d=\Lambda'\cap B^d$.
\qed
\end{proof}

Let $r_1$ and $r_2$ be two nonnegative real numbers such that $r_1 \leq r_2$.
We use $\Sh(r_1,r_2)$ to denote the set $\{x \in \mathbb{R}^d \colon r_1 \leq \|x\| < r_2 \}$.
That is, $\Sh(r_1,r_2)$ is the \emph{spherical shell} bounded by $r_1$ and $r_2$.
The number $r_2-r_1$ is the \emph{width of $\Sh(r_1,r_2)$}.
Note that $\Sh(r_1,r_2)$ is empty if $r_1=r_2$.
Observe that if $r_1 \leq \cdots \leq r_m$ are some nonnegative real numbers, then the shells $\Sh(0,r_1),\Sh(r_1,r_2),\dots,\Sh(r_{m-1},r_m)$ partition the interior of~$B^d(r_m)$.

For $i \in [q-1]$, we let $S_i$ be the set $D' \cap \Sh(\mu_i,\mu_{i+1})$.
Furthermore, we use $S_q$ to denote the set of points from $D'$ that are contained in the closure of the spherical shell $\Sh(\mu_q,qc\alpha)$.

The sets $S_1,\dots,S_q$ then partition $D'$, as there are no points of $\Lambda^*\setminus\{0\}$ in the interior of $B^d(\mu_1)$ from the definition of $\mu_1$ and $\|z\| \leq qc\alpha$ for every $z \in D$. 
We thus have $|D'| = |S_1| + \cdots + |S_q|$.

Let $z$ be an arbitrary element from $D'$.
By Corollary~\ref{cor:indukceNadrovina}, the set $\Lambda_{H(z)} \cap B^d$ can be covered with 
\[c(z) \leq O_{d,k}\left(\left(\frac{\mu_1 \cdots \mu_q}{\|z\|}\right)^{(d-k-1)/(q-2)}\right)\]
$k$-dimensional linear subspaces. 
Since the hyperplanes $H(z)$ with $z \in D'$ cover $\Lambda \cap B^d$, the total number of $k$-dimensional subspaces
needed to cover $\Lambda\cap B^d$ is at most $\sum_{z \in D'} c(z) = \sum_{i=1}^q \sum_{z \in S_i} c(z)$.

To finish the proof of Theorem~\ref{thm:upper}, we show $\sum_{i=1}^q\sum_{z \in S_i}c(z) \leq O_{d,k}(\alpha^{d-k})$.

For $i=1$, we have  $|S_1| \leq 1$, as, by the definitions of~$\mu_1$ and $\mu_2$, the set $S_1$ contains only points from $D$ that lie in $\lin(\{w_1\})$ and the only primitive point satisfying these conditions is $w_1$.
Moreover, every $z \in S_1$ satisfies 
\[c(z) \leq O_{d,k}((\mu_1\cdots\mu_q)^{(d-k-1)/(q-2)}) \leq O_{d,k}\left(\alpha^{(d-k-1)(q-1)/(q-2)}\right),\] since $\|z\|\geq \mu_1 \geq 1$ and $\alpha^{q-1}=\Theta_{d,k}(\mu_1\cdots\mu_q)$.
Since $q \geq d-k+1$, we have 
\[
\frac{(d-k-1)(q-1)}{q-2} \leq \left(1+\frac{1}{q-2}\right)(d-k-1) \leq d-k
\]
and thus $\sum_{z \in S_1}c(z) \leq O_{d,k}(\alpha^{d-k})$.

For $i \in \{2,\dots,q-1\}$, we further refine every set $S_i$ that is determined by a spherical shell of width larger than $1$ into sets $S^1_i,\dots,S^{r_i}_i$, each determined by a spherical shell of width in $(0,1]$, for some positive integer $r_i \leq \mu_{i+1}$.
Such refinement exists, as the Euclidean norm of every vector from $S_i$ is at most $\mu_{i+1}$.
Similarly, if the width of the spherical shell of $S_q$ is larger than $1$, we refine $S_q$ into sets $S_q^1,\dots,S^{r_q}_q$, each determined by a spherical shell of width in $(0,1]$, for some positive integer $r_q \leq qc\alpha$.
We set $S^1_i \colonequals S_i$ and $r_i\colonequals 1$ if the width of $S_i$ is at most $1$.
For all $i \in \{2,\dots,q\}$ and $j \in [r_i]$, let $l_{i,j}$ be the supremum of $\|x\|$ taken over all points $x$ from the spherical shell that determines $S_i^j$.
If the spherical shell is empty, we set $l_{i,j}\colonequals \mu_{i+1}$.
We have $l_{i,j} \in [\mu_i,\mu_{i+1}]$ from the definition of $S^j_i$.
Since the width of every spherical shell $S_i^j$ is at most $1$ and every $z\in D'$ satisfies $\|z\| \geq \mu_1 \geq 1$, every point $z$ from $S^j_i$ also satisfies $l_{i,j}/2 \leq \|z\| \leq l_{i,j}$.
From $\mu_1 \geq 1$, we also have $l_{i,j} \geq 1$.

For every $i\in\{2,\dots,q\}$, the sets $S^1_i,\dots,S^{r_i}_i$ partition $S_i$ and thus we have $|S_i| = \sum_{j=1}^{r_i}|S^j_i|$.
To simplify the notation, we let 
\[c_{i,j} \colonequals \left(\frac{\mu_1 \cdots \mu_q}{l_{i,j}}\right)^{(d-k-1)/(q-2)}\]
for every $i \in \{2,\dots,q\}$ and $j \in [r_i]$.
Then 
\begin{equation}
\label{eq:slupky1}
\sum_{z \in S_i}c(z) = \sum_{j=1}^{r_i} \sum_{z \in S^j_i}c(z) \leq O_{d,k}\left(\sum_{j=1}^{r_i} |S^j_i|c_{i,j}\right)
\end{equation}
for every $i\in \{2,\dots,q\}$.
We show that
\begin{equation}
\label{eq:slupky2}
\sigma \colonequals \sum_{i=2}^q\sum_{j=1}^{r_i} |S^j_i|c_{i,j} \leq 2^q2^{3d}\sum_{i=2}^q \sum_{j=1}^{r_i} \frac{l_{i,j}^{i-1}}{\mu_1\cdots\mu_i}c_{i,j} + O_{d,k}(\alpha^{d-k}).
\end{equation}

By Theorem~\ref{thm:MinkowskiEnumerator}, we have $|\Lambda^* \cap B^d(l_{i,j})| \leq 2^{d-1}\prod_{m=1}^d \lfloor 2/\lambda_m(\Lambda^*,B^d(l_{i,j})) +1\rfloor$  for every  $i \in \{2,\dots,q\}$ and $j \in [r_i]$. 
From the definition of $\mu_1,\dots,\mu_d$, we have $\lambda_m(\Lambda^*,B^d(l_{i,j})) = \mu_m/ l_{i,j}$.
Thus 
\begin{align*}|\Lambda^* \cap B^d(l_{i,j})| \leq 2^{d-1}\prod_{m=1}^d \left\lfloor \frac{2 l_{i,j}}{\mu_m} +1\right\rfloor \leq 2^{3d-2i}\prod_{m=1}^i \left\lfloor \frac{2l_{i,j}}{\mu_m} +1\right\rfloor 
\leq 2^{3d} \frac{l_{i,j}^i}{\mu_1\cdots\mu_i},
\end{align*}
where the second inequality follows from $l_{i,j} \leq \mu_{i+1}\leq \cdots \leq \mu_d$ and the last inequality from $\mu_1 \leq \cdots \leq \mu_i \leq l_{i,j}$.

Let $i$ be an integer from $\{2,\dots,q\}$ and $j$ be an integer from $[r_i]$.
For a nonnegative real number~$r$, let $B^d_o(r)$ be the open $d$-dimensional ball centered in the origin with radius $r$.
We have $|S_i^j| \leq |\Lambda^* \cap B^d(l_{i,j})| - |D' \cap B^d_o(l_{i,j-1})|$ if $j>1$ and $|S_i^j| \leq |\Lambda^* \cap B^d(l_{i,j})| - |D' \cap B^d_o(l_{i-1,r_{i-1}})|$ otherwise.
In the first case, we assume $|D' \cap B^d_o(l_{i,j-1})| \geq 2^{3d}l_{i,j-1}^i/(\mu_1\cdots\mu_i)$, as otherwise, to show~\eqref{eq:slupky2}, we may take $2^{3d}l_{i,j-1}^i/(\mu_1\cdots\mu_i) - |D'\cap B^d_o(l_{i,j-1})|$ points from $\Lambda^* \cap (B^d(qc\alpha) \setminus B_o^d(l_{i,j-1}))$ (or less if there are not that many points) and add $c_{i,j-1}$ to $\sigma$ for each one of them instead of adding at most $c_{i,j}$.
This will still bound $\sigma$ from above, as $c_{i,j-1} \geq c_{i,j}$.
Thus we obtain
\[|S^j_i| \leq 2^{3d}\frac{(l_{i,j}^i -  l^i_{i,j-1})}{\mu_1\cdots\mu_i} \leq 2^q2^{3d}\frac{l_{i,j}^{i-1}}{\mu_1\cdots\mu_i},\]
 where the second inequality follows from $l_{i,j-1}\geq l_{i,j} - 1$ and $(a^i-(a-1)^i) \leq 2^ia^{i-1} \leq 2^qa^{i-1}$ for $a \geq 1$.

In the other case, $j=1$ and $|S_i^1| \leq |\Lambda^* \cap B^d(l_{i,1})| - |D' \cap B^d_o(l_{i-1,r_{i-1}})|$.
If $i\ge 3$, we apply the same argument as in the first case, so we can assume $|D' \cap B^d_o(l_{i-1,r_{i-1}})| \geq 2^{3d}  l_{i-1,r_{i-1}}^{i-1}/(\mu_1\cdots\mu_{i-1})$, since $c_{i-1,r_{i-1}} \geq c_{i,1}$. 
Thus 
\[|S_i^1| \leq 2^{3d}\left(\frac{l_{i,1}^i}{\mu_1\cdots\mu_i} -\frac{ l_{i-1,r_{i-1}}^{i-1}}{\mu_1\cdots\mu_{i-1}}\right)
 = 2^{3d}\left(\frac{l_{i,1}^i}{\mu_1\cdots\mu_i} -\frac{l_{i-1,r_{i-1}}^i}{\mu_1\cdots\mu_i}\right) 
\leq 2^q2^{3d}\frac{l_{i,1}^{i-1}}{\mu_1\cdots\mu_i},\]
where the equality follows from $l_{i-1,r_{i-1}}=\mu_i$ and the inequality from $l_{i-1,r_{i-1}} \geq  l_{i,1}-1$.

If $j=1$ and $i=2$, we use
\[
|S_2^1| \le |\Lambda^* \cap B^d(l_{2,1})| \le 2^{3d} \frac{l_{2,1}^2}{\mu_1 \mu_2}
\]
and thus
\begin{align*}
|S_2^1| c_{2,1} 
&\le 2^{3d} \frac{l_{2,1}^2}{\mu_1 \mu_2} \left(\frac{\mu_1 \cdots \mu_q}{l_{2,1}}\right)^{(d-k-1)/(q-2)} \\
&\in O_{d,k} \left(\frac{\mu_2^2}{\mu_1 \mu_2} \left(\frac{\mu_1 \cdots \mu_q}{\mu_2}\right)^{(d-k-1)/(q-2)} \right)
\qquad \text{since $l_{2,1} \in \Theta(\mu_2)$} \\
&\in O_{d,k} \left(\mu_2 \cdot \left(\mu_3 \cdots \mu_q\right)^{(d-k-1)/(q-2)} \right) 
\qquad \text{since $d-k+1 \le q$ and $\mu_1 \ge 1$} \\
&\in O_{d,k}(\alpha^{d-k}) \qquad \text{since $\mu_m \in O_{d,k}(\alpha)$ for every $m \in [q]$ by~\eqref{eq:minimaRevers} and part~\ref{item-alpha1} of Lemma~\ref{lem:lambdaVsAlpha}.}
\end{align*}

This gives us the inequality~\eqref{eq:slupky2}.

By~\eqref{eq:slupky1} and~\eqref{eq:slupky2}, it remains to prove that the right side of~\eqref{eq:slupky2} is at most $O_{d,k}(\alpha^{d-k})$.
We do so by estimating each term $\sigma_i \colonequals \sum_{j=1}^{r_i} \frac{l_{i,j}^{i-1}}{\mu_1\cdots\mu_i}c_{i,j}$ with $O_{d,k}(\alpha^{d-k})$ for $i=2,\dots,q$.

For $i=q$, we have 
\[\sigma_q = \sum_{j=1}^{r_q} \frac{l_{q,j}^{q-1}}{\mu_1 \cdots \mu_q}\left(\frac{\mu_1 \cdots \mu_q}{l_{q,j}}\right)^{(d-k-1)/(q-2)}
= \sum_{j=1}^{r_q} \frac{l_{q,j}^{q-1-(d-k-1)/(q-2)}}{\mu_1 \cdots \mu_q}(\mu_1 \cdots \mu_q)^{(d-k-1)/(q-2)}.\]
We have $q-1-(d-k-1)/(q-2) \geq 1$ from $q \geq d-k+1$ and $q > 2$.
Moreover, $\alpha^{q-1} = \Theta_{d,k}(\mu_1\cdots\mu_q)$, $l_{q,j} \leq qc\alpha$, and $r_q \leq qc\alpha$.
Thus
\[\sigma_q \leq O_{d,k}\left(\frac{\alpha^{q-(d-k-1)/(q-2)}}{\alpha^{q-1}}(\alpha^{q-1})^{(d-k-1)/(q-2)}\right)
= O_{d,k}\left(\frac{\alpha^{q}}{\alpha^{q-1}}(\alpha^{q-2})^{(d-k-1)/(q-2)}\right)
= O_{d,k}\left(\alpha^{d-k}\right).\]

For $2\leq i<q$, we obtain 
\[\sigma_i = \sum_{j=1}^{r_i} \frac{l_{i,j}^{i-1}}{\mu_1 \cdots \mu_i} \left(\frac{\mu_1 \cdots \mu_q}{l_{i,j}}\right)^{(d-k-1)/(q-2)}
= \sum_{j=1}^{r_i} \frac{l_{i,j}^{i-1 - (d-k-1)/(q-2)}}{\mu_1 \cdots \mu_i} 
(\mu_1 \cdots \mu_q)^{(d-k-1)/(q-2)}.\]

Since $q \ge d-k+1$ and $i \geq 2$, we have $i-1 - (d-k-1)/(q-2) \geq 0$.
Using $l_{i,j} \leq \mu_{i+1}$ and $r_i \leq \mu_{i+1}$, the summation gives
\[\sigma_i \leq \frac{\mu_{i+1}^{i - (d-k-1)/(q-2)}}{\mu_1 \cdots \mu_i} (\mu_1 \cdots \mu_q)^{(d-k-1)/(q-2)} 
= \frac{\mu_{i+1}^i(\mu_{i+2} \cdots \mu_q)^{(d-k-1)/(q-2)}}
{(\mu_1 \cdots \mu_i)^{(q-d+k-1)/(q-2)}}.\]

By~\eqref{eq:minimaRevers} and part~\ref{item-alpha1} of Lemma~\ref{lem:lambdaVsAlpha}, $\mu_m \leq O_{d,k}(\alpha)$ for every $m \in [q]$. 
This fact together with $\alpha^{q-1}=\Theta_{d,k}(\mu_1\cdots\mu_q)$ gives 
\[\mu_1 \cdots \mu_i 
= \frac{\mu_1 \cdots \mu_q}{\mu_{i+1} \cdots \mu_q} 
\geq \Omega_{d,k}\left(\frac{\alpha^{q-1}}{\alpha^{q-i}}\right)
= \Omega_{d,k}\left(\alpha^{i-1}\right).\]
Thus we have
\[
\sigma_i \le O_{d,k}\left(\frac{\alpha^{i + (q-i-1)(d-k-1)/(q-2)}}{\alpha^{(i-1)(q-d+k-1)/(q-2)}}\right)
= O_{d,k}\left(\alpha^{i + ((q-i-1)(d-k-1) - (i-1)(q-d+k-1))/(q-2)}\right).
\]

To simplify the expression in the exponent of $\alpha$, we rewrite
\begin{align*}
(q-i-1)(d-k-1) - (i-1)(q-d+k-1) 
&= (q-i-1)(d-k-1) + (i-1)((d-k-1)+(2-q)) \\
&= (q-2)(d-k-1) - (i-1)(q-2) \\
&= (q-2)(d-k-i).
\end{align*}

Now we have
\[\sigma_i \le O_{d,k}\left(\alpha^{i+d-i-k}\right)=O_{d,k}\left(\alpha^{d-k}\right).\]

This finishes the proof of Theorem~\ref{thm:upper}.

\subsection{The general case}

Here, we finish the proof of Theorem~\ref{thm:upperGeneral} by extending Theorem~\ref{thm:upper} to arbitrary convex bodies from $\mathcal{K}^d$.
This is done by approximating a given body $K$ from $\mathcal{K}^d$ with ellipsoids.
A \emph{$d$-dimensional ellipsoid} in $\mathbb{R}^d$ is an image of $B^d$ under a nonsingular affine map.
Such approximation exists by the following classical result, called \emph{John's lemma}~\cite{john48}.

\begin{lemma}[{John's lemma~\cite[see Theorem 13.4.1]{mat02}}]
\label{lem:John}
For every positive integer $d$ and every $K \in \mathcal{K}^d$, there is a $d$-dimensional ellipsoid $E$ with the center in the origin that satisfies
\[E/\sqrt{d} \subseteq K\subseteq E.\]
\end{lemma}

Let $\Lambda \in \mathcal{L}^d$ be a given lattice and let $\lambda_i \colonequals \lambda_i(\Lambda,K)$ for every $i\in[d]$.
From our assumptions, we know that $\lambda_d \leq 1$.
Let $E$ be the ellipsoid from Lemma~\ref{lem:John}.
Since $E$ is an ellipsoid, there is a nonsingular affine map $h \colon \mathbb{R}^d \to \mathbb{R}^d$ such that $E=h(B^d)$.
Since $E$ is centered in the origin, we see that $h$ is in fact a linear map.
Thus $\Lambda'\colonequals h^{-1}(\Lambda) \in \mathcal{L}^d$.
Observe that we have $\lambda_i = \lambda_i(\Lambda',h^{-1}(K))$ for every $i \in [d]$.

For every $i \in [d]$, we use $\lambda'_i$ to denote the $i$th successive minimum $\lambda_i(\Lambda',B^d)=\lambda_i(\Lambda,E)$.
From the choice of $E$, we have $\lambda_i/\sqrt{d} \leq \lambda'_i \leq \lambda_i$.
In particular, $\lambda'_d \leq 1$.
Thus, by Theorem~\ref{thm:upper}, the set $\Lambda' \cap B^d$ can be covered with $O_{d,k}((\alpha')^{d-k})$ $k$-dimensional linear subspaces, where $\alpha' \colonequals \min_{1 \leq j \leq k}(\lambda'_j\cdots\lambda'_d)^{-1/(d-j)}$.

Since $\lambda_i = \Theta_d(\lambda'_i)$ for every $i \in [d]$, we see that the set $\Lambda' \cap h^{-1}(K)$ can be covered with $O_{d,k}(\alpha^{d-k})$ $k$-dimensional linear subspaces, where $\alpha \colonequals \min_{1 \leq j \leq k}(\lambda_j\cdots\lambda_d)^{-1/(d-j)}$.
Since every nonsingular linear transformation preserves incidences and successive minima and maps a $k$-dimensional linear subspace to a $k$-dimensional linear subspace, the set $\Lambda \cap K$ can be covered with $O_{d,k}(\alpha^{d-k})$ $k$-dimensional linear subspaces.

\section{Proof of Theorem~\ref{thm:lower}}
\label{sec:thmLower}

Let $d$ and $k$ be positive integers satisfying $1 \leq k \leq d-1$ and let $K$ be a body from $\mathcal{K}^d$ with $\lambda_d(\mathbb{Z}^d,K) \leq 1$.
For every $i \in [d]$, we let $\lambda_i$ be the $i$th successive minimum $\lambda_i(\mathbb{Z}^d,K)$.
Let $\varepsilon$ be a number from $(0,1)$.
We use a probabilistic approach to show that there is a set $S \subseteq \mathbb{Z}^d \cap K$ of size at least $\Omega_{d,\varepsilon,k}(((1-\lambda_d)\beta)^{d-k-\varepsilon})$, where
\[\beta \colonequals  \min_{1 \leq j \leq d-1}(\lambda_j\cdots\lambda_d)^{-1/(d-j)},\]
 such that every $k$-dimensional linear subspace contains at most $r-1$ points from $S$.

Note that it is sufficient to prove the statement only for the lattice $\mathbb{Z}^d$. 
For a general lattice $\Lambda \in \mathcal{L}^d$ we can apply a linear transformation $h$ such that $h(\Lambda) = \mathbb{Z}^d$ and then use the result for $\mathbb{Z}^d$ and $h(K)$, since $\lambda_i(\Lambda,K)=\lambda_i(\mathbb{Z}^d,h(K))$ for every $i \in [d]$.
We also remark that in the case $k=d-1$ the stronger lower bound $\Omega_d((1-\lambda_d)\beta)$  from Theorem~\ref{thm:Barany} by B\'{a}r\'any et al.~\cite{bhpt01} applies.

The proof is based on the following two results, first of which is by B\'{a}r\'any et al.~\cite{bhpt01}.

\begin{lemma}[\cite{bhpt01}]
\label{lem:lowerBarany}
For an integer $d \geq 2$ and $K \in \mathcal{K}^d$, if $\lambda_d < 1$ and $p$ is an integer satisfying $1 < p <(1-\lambda_d)\beta/(8d^2)$, then, for every $v \in \mathbb{R}^d$, there exist an integer $1 \leq j < p$ and a point $w \in \mathbb{Z}^d$ with $jv + pw \in K$.
\end{lemma}

For a prime number $p$, let $\mathbb{F}_p$ be the finite field of size $p$.
The second main ingredient in the proof of Theorem~\ref{thm:lower} is the following lemma.

\begin{lemma}
\label{lem:lowerProbabilistic}
Let $d$ and $k$ be integers satisfying $2 \leq k \leq d-2$ and let $\varepsilon \in (0,1)$.
Then there is a positive integer $p_0 = p_0(d,\varepsilon,k)$ such that for every prime number $p \geq p_0$ there exists a subset $R$ of $\mathbb{F}_p^{d-1}$ of size at least $ p^{d-k-\varepsilon}/2$ such that every $(k-1)$-dimensional affine subspace of $\mathbb{F}_p^{d-1}$ contains at most $r-1$ points from $R$ for $r \colonequals \lceil k(d-k+1)/\varepsilon \rceil$. 
\end{lemma}
\begin{proof}
We assume that $p$ is large enough with respect to $d$, $\varepsilon$, and $k$ so that $p^{k-1} > r$.
We set $P \colonequals p^{1-k-\varepsilon}$ and we let $X$ be a subset of $\mathbb{F}^{d-1}_p$ obtained by choosing every point from~$\mathbb{F}^{d-1}_p$ independently at random with the probability $P$.

Let $A$ be a $(k-1)$-dimensional affine subspace of $\mathbb{F}^{d-1}_p$.
Then $|A|=p^{k-1}$.
It is well-known that the number of $(k-1)$-dimensional linear subspaces of $\mathbb{F}^{d-1}_p$ is exactly the \emph{Gaussian binomial coefficient}
\begin{align}
\label{eq:Gaussian}
\begin{split}
\left[\genfrac{}{}{0pt}{}{d-1}{k-1}\right]_p &\colonequals \frac{(p^{d-1}-1)(p^{d-1}-p)\cdots(p^{d-1}-p^{k-2})}{(p^{k-1}-1)(p^{k-1}-p)\cdots(p^{k-1}-p^{k-2})} \\
& \leq \frac{p^{d-1}\cdot p^{d-2}\cdots p^{d-k+1}}{(p^{k-1}-1)(p^{k-2}-1)\cdots(p-1)} \leq p^{(k-1)d - (k-1)k/2 - (k-1)(k-2)/2} = p^{(k-1)(d-k+1)}.
\end{split}
\end{align}
We used the fact $p^{k-i}-1 \geq p^{k-i-1}$ for $k > i$ in the last inequality. 

Since every $(k-1)$-dimensional affine subspace $A$ of $\mathbb{F}^{d-1}_p$ is of the form $A=x+L$ for some $x \in \mathbb{F}^{d-1}_p$ and a $(k-1)$-dimensional linear subspace $L$ of $\mathbb{F}^{d-1}_p$ and $x+L=y+L$ if and only if $x-y \in L$, the total number of $(k-1)$-dimensional affine subspaces of $\mathbb{F}^{d-1}_p$ is $p^{d-k}\left[\genfrac{}{}{0pt}{}{d-1}{k-1}\right]_p$.
This is because by considering pairs $(x,L)$, where $x \in \mathbb{F}_p^{d-1}$ and $L$ is a $(k-1)$-dimensional linear subspace of $\mathbb{F}_p^{d-1}$, every $(k-1)$-dimensional affine subspace $A$ is counted $p^{k-1}$ times.

We use the following Chernoff-type bound (see the last bound of~\cite{hagRu90}) to estimate the probability that $A$ contains at least $r$ points of $X$.
Let $q \in [0,1]$ and let $Y_1,\dots,Y_m$ be independent 0-1 random variables with $\Pr[Y_i=1]=q$ for every $i \in [m]$.
Then, for $mq \leq s < m$, we have
\begin{equation}
\label{eq:Chernoff}
\Pr[Y_1+\cdots+Y_m \geq s] \leq \left(\frac{mq}{s}\right)^s e^{s-mq}.
\end{equation}

Choosing $Y_x$ as the indicator variable for the event $x \in A \cap X$ for each $x \in A$, we have $m=|A|=p^{k-1}$ and $q=P$.
Since $p,r \ge 1$ and $p^{k-1} > r$, we have $p^{-\varepsilon}= mq \leq r < m = p^{k-1}$ and thus the bound~\eqref{eq:Chernoff} implies
\[\Pr[|A \cap X| \geq r] \leq \left(\frac{p^{k-1}P}{r}\right)^re^{r-p^{k-1}P} =  \left(\frac{p^{-\varepsilon}}{r}\right)^re^{r-p^{-\varepsilon}} = p^{-\varepsilon r}e^{r(1 - \ln{r}) - p^{-\varepsilon}} < p^{-\varepsilon r},\]
where the last inequality follows from $r \geq e$, as then $1-\ln{r} \leq 0$.

By the union bound, the probability that there is a $(k-1)$-dimensional  affine subspace $A$ of $\mathbb{F}^{d-1}_p$ with $|A \cap X| \geq r$ is less than
\[p^{d-k}\left[\genfrac{}{}{0pt}{}{d-1}{k-1}\right]_p \cdot p^{-\varepsilon r} \leq p^{(d-k) + (k-1)(d-k+1) - \varepsilon r} \leq p^{k(d-k+1)-1 - k(d-k+1)}=p^{-1},\]
where the first inequality follows from~\eqref{eq:Gaussian} and the second inequality is due to the choice of $r$.
From $p \geq 2$, we see that this probability is less than $1/2$. 

The expected size of $X$ is \[\mathbb{E}[|X|] = |\mathbb{F}^{d-1}_p|\cdot P = p^{d-1}p^{1-k-\varepsilon} = p^{d-k-\varepsilon}.\]
Since $|X| \sim {\rm Bi}(p^{d-1},P)$, the variance of $|X|$ is $p^{d-1}P(1-P) < p^{d-k-\varepsilon}$ and Chebyshev's inequality implies $\Pr[||X|-\mathbb{E}[|X|]| \geq \sqrt{2p^{d-k-\varepsilon}}] < p^{d-k-\varepsilon}/(2p^{d-k-\varepsilon}) =  1/2$.

Thus there is a set $R$ of size at least $p^{d-k-\varepsilon}-\sqrt{2p^{d-k-\varepsilon}} \geq p^{d-k-\varepsilon}/2$ such that every $(k-1)$-dimensional affine subspace of $\mathbb{F}^{d-1}_p$ contains at most $r-1$ points from~$R$.
\qed
\end{proof}

Let $\varepsilon \in (0,1)$ be given.
To derive Theorem~\ref{thm:lower}, we combine Lemma~\ref{lem:lowerBarany} with Lemma~\ref{lem:lowerProbabilistic}.
This is a similar approach as in~\cite{bhpt01}, where the authors derive a lower bound for the case $k=d-1$ by combining Lemma~\ref{lem:lowerBarany} with a construction found by Erd\H{o}s in connection with Heilbronn's triangle problem~\cite{roth51}.

Let $p$ be the largest prime number that satisfies the assumptions of Lemma~\ref{lem:lowerBarany}.
If such $p$ does not exist, then the statement of the theorem is trivial.
By Bertrand's postulate, we have $p > (1-\lambda_d)\beta/(16d^2)$.
We may assume that $p \geq p_0$, where $p_0=p_0(d,\varepsilon,k)$ is the constant from Lemma~\ref{lem:lowerProbabilistic}, since otherwise 
the statement of Theorem~\ref{thm:lower} is trivial.

For $k \geq 2$ and $t \colonequals \lceil p^{d-k-\varepsilon}/2\rceil$, let $R=\{v_1,\dots,v_t\} \subseteq \mathbb{F}^{d-1}_p$ be the set of points from Lemma~\ref{lem:lowerProbabilistic}.
That is, every $(k-1)$-dimensional affine subspace of $\mathbb{F}^{d-1}_p$ contains at most $r-1$ points from $R$ for $r\colonequals \lceil k(d-k+1)/\varepsilon \rceil$.
In particular, every $r$-tuple of points from $R$ contains $k+1$ affinely independent points over the field~$\mathbb{F}_p$.
For $k=1$, we can set $r \colonequals 2$ and let $R$ be the whole set $\mathbb{F}_p^{d-1}$ of size $t \colonequals p^{d-k}=p^{d-1}$.
Then every $r$-tuple of points from $R$ contains two affinely independent points over the field $\mathbb{F}_p$.

For $i=1,\dots,t$, let $u_i \in \mathbb{Z}^d$ be the vector obtained from $v_i$ by adding $1$ as the last coordinate.
From the choice of $R$, every $r$-tuple of points from $\{u_1,\dots,u_t\}$ contains $k+1$ points that are linearly independent over the field $\mathbb{F}_p$.

By Lemma~\ref{lem:lowerBarany}, there exist an integer $1 \leq j_i < p$ and a point $w_i \in \mathbb{Z}^d$ for every $i \in [t]$ such that $u'_i \colonequals j_iu_i + pw_i$ lies in $K$.
We have $u'_i \equiv j_iu_i \pmod{p}$ for every $i \in [t]$ and thus every $r$-tuple of vectors from $S \colonequals \{u'_1,\dots,u'_t\} \subseteq \mathbb{Z}^d$ contains $k+1$ linearly independent vectors over the field $\mathbb{F}_p$, and hence over $\mathbb{R}$.
In other words, every $k$-dimensional linear subspace of $\mathbb{R}^d$ contains at most $r-1$ points from~$S$.
Since $|S|=t = \lceil p^{d-k-\varepsilon}/2\rceil$ and $p > (1-\lambda_d)\beta/(16d^2)$, we have $l(d,k,n,r) \geq \Omega_{d,k}(((1-\lambda_d)\beta)^{d-k-\varepsilon})$.
This completes the proof of Theorem~\ref{thm:lower}.

\section{Proof of Theorem~\ref{thm:coveringAffine}}
\label{sec:thmCoveringAffine}

Let $d$ and $k$ be integers with $1 \leq k \leq d-1$ and let $\Lambda \in \mathcal{L}^d$ and $K \in\mathcal{K}^d$.
We let $\lambda_i\colonequals\lambda_i(\Lambda,K)$ for every $i \in [d]$ and assume that $\lambda_d \leq 1$.
First, we observe that it is sufficient to prove the statement only for $K=B^d$, as we can then strengthen the statement to an arbitrary $K \in \mathcal{K}^d$ using John's lemma (Lemma~\ref{lem:John}) analogously as in the proof of Theorem~\ref{thm:upperGeneral}.

First, we prove the upper bound.
That is, we show that $\Lambda \cap B^d$ can be covered with $O_{d,k}((\lambda_{k+1} \cdots \lambda_d)^{-1})$ $k$-dimensional affine subspaces of $\mathbb{R}^d$.
By Lemma~\ref{lem:projection}, there is a positive integer $r=r(d,k)$ and a projection $p$ of $\mathbb{R}^d$ along $k$ vectors $b_1,\dots,b_k$ from $\Lambda$ onto a $(d-k)$-dimensional linear subspace $N$ of $\mathbb{R}^d$ such that $\Lambda \cap B^d$ is mapped to $\Lambda \cap N \cap B^d(r)$ and such that $\lambda'_i \colonequals \lambda_i(\Lambda \cap N, B^d(r) \cap N) = \Theta_{d,k}(\lambda_{i+k})$ for every $i \in [d-k]$.

For each point $z$ of $\Lambda \cap N \cap B^d(r)$, we define $A(z)$ to be the affine hull of the set $\{z,b_1+z,\dots,b_k+z\}$.
Every $A(z)$ is then a $k$-dimensional affine subspace of $\mathbb{R}^d$ and the set $\mathcal{A} \colonequals \{A(z) \colon z \in \Lambda \cap N \cap B^d(r)\}$ covers $\Lambda \cap B^d$, since $p(z) \in \Lambda \cap N \cap B^d(r)$ for every $z \in \Lambda \cap B^d$.
We have $|\mathcal{A}| = |\Lambda \cap N \cap B^d(r)|$ and, since $\lambda_d \leq 1$ and $\lambda'_1 \leq \cdots\leq\lambda'_{d-k} \leq O_{d,k}(\lambda_d)$, Theorem~\ref{thm:MinkowskiEnumerator} implies $ |\Lambda \cap N \cap B^d(r)| \leq O_{d,k}((\lambda'_1 \cdots \lambda'_{d-k})^{-1})$.
The bound $\lambda'_i \geq \Omega_{d,k}(\lambda_{i+k})$ for every $i \in [d-k]$ then gives $|\mathcal{A}| \leq O_{d,k}((\lambda_{k+1}\cdots\lambda_d)^{-1})$.

To show the lower bound, we prove that we need at least $\Omega_{d,k}((\lambda_{k+1}\cdots\lambda_d)^{-1})$  $k$-dimensional affine subspaces of $\mathbb{R}^d$ to cover $\Lambda \cap B^d$.

Let $A$ be a $k$-dimensional affine subspace of $\mathbb{R}^d$.
We show that $A$ contains at most $O_{d,k}((\lambda_1\cdots\lambda_k)^{-1})$ points from $\Lambda \cap B^d$.
Let $y$ be an arbitrary point from $\Lambda \cap A  \cap B^d$.
Then $A = L + y$, where $L$ is a $k$-dimensional linear subspace of $\mathbb{R}^d$, and $(\Lambda \cap A)-y = \Lambda \cap L$.
For every $i \in [k]$, we let $\lambda'_i \colonequals \lambda_i(\Lambda \cap L, B^d(2))$ and we observe that $\lambda'_i \geq \lambda_i/2$.
By Theorem~\ref{thm:MinkowskiEnumerator}, we have $|\Lambda \cap L \cap B^d(2)| \leq O_{d,k}((\lambda'_1 \cdots \lambda'_s)^{-1})$, where $s$ is the maximum integer $j$ from $[k]$ with $\lambda'_j \leq 1$.
Since $\lambda'_i \geq  \lambda_i/2$ for every $i \in [k]$, we have $|\Lambda \cap L \cap B^d(2)| \leq O_{d,k}((\lambda_1\cdots \lambda_k)^{-1})$.
For every $x \in A \cap B^d$, we have $\|x-y\| \leq \|x\| + \|y\| \leq 2$ and thus $x-y \in L \cap B^d(2)$.
It follows that $(\Lambda \cap A \cap B^d)-y \subseteq \Lambda \cap L \cap B^d(2)$ and thus $|\Lambda \cap A \cap B^d| \leq O_{d,k}((\lambda_1\cdots\lambda_k)^{-1})$.

Let $\mathcal{A}$ be a collection of $k$-dimensional affine subspaces of~$\mathbb{R}^d$ that covers $\Lambda \cap B^d$.
We have $|\mathcal{A}| \geq |\Lambda \cap B^d|/m$, where $m$ is the maximum of $|\Lambda \cap A \cap B^d|$ taken over all subspaces $A$ from $\mathcal{A}$.
We know that $m \leq O_{d,k}((\lambda_1\cdots\lambda_k)^{-1})$.
It is a well-known fact that follows from Minkowski's second theorem (Theorem~\ref{thm:2ndMinkowski}) that $|\Lambda \cap B^d| \geq \Omega_{d,k}((\lambda_1\cdots\lambda_d)^{-1})$.
Thus we obtain
\[|\mathcal{A}| \geq \frac{|\Lambda \cap B^d|}{m} \geq \frac{\Omega_{d,k}((\lambda_1\cdots\lambda_d)^{-1})}{O_{d,k}((\lambda_1\cdots\lambda_k)^{-1})} \geq \Omega_{d,k}((\lambda_{k+1}\cdots\lambda_d)^{-1}),\]
which finishes the proof of  Theorem~\ref{thm:coveringAffine}.

\section{Proofs of Theorems~\ref{thm:incidence} and~\ref{thm:incidenceNondiagonal}}

Assume that we are given integers $d$ and $k$ with $0 \leq k \leq d-2$ and let $\varepsilon$ be a real number in $(0,1)$.
Let $\delta=\delta(d,\varepsilon,k) \in (0,1)$ be a sufficiently small constant.
By~\eqref{eq:affine}, there is a positive integer $r_1=r_1(d,\delta,k)$ and a constant $c_1=c_1(d,\delta,k)$ such that for every $s \in \mathbb{N}$ there is a subset $P$ of $\mathbb{Z}^d \cap B^d(s)$ of size $c_1 \cdot s^{d-k-\delta}$ such that every $k$-dimensional affine subspace of $\mathbb{R}^d$ contains at most $r_1-1$ points from $P$.
 In the case $k=0$, we can clearly obtain the stronger bound $c_1 \cdot s^d$.

By Corollary~\ref{cor:lower}, there is a positive integer $r_2=r_2(d,\delta,k)$ and a constant $c_2=c_2(d,\delta,k)$ such that for every $t \in \mathbb{N}$ there is a subset $N'$ of $\mathbb{Z}^d \cap B^d(t)$ of size $c_2 \cdot t^{d(k+1-\delta)/(d-1)}$ such that every $(d-k-1)$-dimensional linear subspace contains at most $r_2-1$ points from $N'$.
In particular, every $1$-dimensional linear subspace contains at most $r_2-1$ points from $N'$ and thus there is a set $N \subseteq N'$ of size $|N|=|N'|/(r_2-1)=c_2 \cdot t^{d(k+1-\delta)/(d-1)}/(r_2-1)$ containing only primitive vectors.
We note that for $k=0$ we can apply Theorem~\ref{thm:Barany} instead of Corollary~\ref{cor:lower} and obtain the stronger bound $|N| = c_2 \cdot t^{d/(d-1)}/(r_2-1)$.
We let $\mathcal{H}$ be the set of hyperplanes in $\mathbb{R}^d$ with normal vectors from $N$ such that every hyperplane from~$\mathcal{H}$ contains at least one point of $P$.

We show that the graph $G(P,\mathcal{H})$ does not contain $K_{r_1,r_2}$.
If there is an $r_2$-tuple of hyperplanes from~$\mathcal{H}$ with a nonempty intersection, then these hyperplanes have distinct normal vectors that span a linear subspace of dimension at least $d-k$ by the choice of $N$.
The intersection of these hyperplanes is thus an affine subspace of dimension at most $k$.
From the definition of $P$, it contains at most $r_1-1$ points from $P$.

We set 
\[n \colonequals c_1 \cdot s^{d-k-\delta}
\hskip 0.5cm \text{ and } \hskip 0.5cm
m \colonequals \frac{3c_2}{r_2-1} \cdot s \cdot t^{d(k+2-1/d-\delta)/(d-1)}.\]
Then we have $|P|=n$.
For every $p \in P$ and $z \in N$, we have $\langle p, z \rangle \in \mathbb{Z}$ and $|\langle p,z \rangle | \leq \|p\|\|z\| \leq s t$ by the Cauchy--Schwarz inequality.
Thus every point $z$ from $N$  is the normal vector of at most $2s t+1 \leq 3st$ hyperplanes from $\mathcal{H}$.
It follows that
\[|\mathcal{H}| \leq 3 s t |N| =3 s t \frac{c_2 \cdot t^{d(k+1-\delta)/(d-1)}}{r_2-1} = \frac{3c_2}{r_2-1} \cdot s \cdot t^{d(k+2-1/d-\delta)/(d-1)} = m.\]
From the definition of $\mathcal{H}$, the number of incidences between $P$ and $\mathcal{H}$ is at least 
\begin{align}
\label{eq:incidence}
\begin{split}
|P||N| &=  n \cdot \frac{c_2 \cdot t^{d(k+1-\delta)/(d-1)}}{r_2-1} = \Omega_{d,\varepsilon,k}\left( n \cdot (m/s)^{(k+1-\delta)/(k+2-1/d-\delta)}\right)  \\
& = \Omega_{d,\varepsilon,k}\left(n^{1-(k+1-\delta)/((k+2-1/d-\delta)(d-k-\delta))}m^{(k+1-\delta)/(k+2-1/d-\delta)}\right) \\
&\geq \Omega_{d,\varepsilon,k}\left(n^{1-(k+1)/((k+2-1/d)(d-k))-\varepsilon}m^{(k+1)/(k+2-1/d)-\varepsilon}\right),
\end{split}
\end{align}
where the last inequality holds for $\delta$ sufficiently small with respect to $d$, $\varepsilon$, and $k$.
This finishes the proof of Theorem~\ref{thm:incidenceNondiagonal}.

To maximize the number of incidences in the diagonal case, we choose $k \colonequals \lfloor \frac{d-2}{2}\rfloor$.
For $d$ odd, we then have at least
\[\Omega_{d,\varepsilon}\left(n^{1-2(d-1)/((d+1-2/d)(d+3))-\varepsilon}m^{(d-1)/(d+1-2/d)-\varepsilon}\right)\]
incidences by~\eqref{eq:incidence}.
By duality, we may obtain a symmetrical expression by averaging the exponents.
Then we obtain
\[\inc(P,\mathcal{H}) \geq \Omega_{d,\varepsilon}\left((mn)^{(d^2+3d+3)/(d^2+5d+6)-\varepsilon}\right) = \Omega_{d,\varepsilon}\left((mn)^{1-(2d+3)/((d+2)(d+3))-\varepsilon}\right).\]

For $d$ even, the choice of $k$ implies that the number of incidences is at least
\[\Omega_{d,\varepsilon}\left(n^{1-2d/((d+2-2/d)(d+2))-\varepsilon}m^{d/(d+2-2/d)-\varepsilon}\right)\]
by~\eqref{eq:incidence}.
Using the averaging argument, we obtain
\begin{align*}
\inc(P,\mathcal{H}) &\geq \Omega_{d,\varepsilon}\left((mn)^{(d^3+2d^2+d-2)/((d+2)(d^2+2d-2))-\varepsilon}\right) \\
&= \Omega_{d,\varepsilon}\left((mn)^{1-(2d^2+d-2)/((d+2)(d^2+2d-2)) -\varepsilon}\right) .
\end{align*}

This completes the proof of Theorem~\ref{thm:incidence}.
For $d \leq 3$, we have $k=0$ and thus we can get rid of the $\varepsilon$ in the exponent by applying the stronger bounds on $m$ and $n$.

\paragraph{Remark.}
An upper bound similar to~\eqref{eq:incidenceUpper} holds in a much more general setting, where we bound the maximum number of edges in $K_{r,r}$-free \emph{semi-algebraic bipartite graphs $G=(P \cup Q,E)$ in $(\mathbb{R}^d,\mathbb{R}^d)$ with bounded description complexity $t$} (see~\cite{fpssz14} for definitions).
Fox, Pach, Sheffer, Suk, and Zahl~\cite{fpssz14} showed that the maximum number of edges in such graphs with $|P|=n$ and $|Q|=m$ is at most $O_{d,\varepsilon,r,t}((mn)^{1-1/(d+1)+\varepsilon}+m+n)$ for an arbitrarily small constant $\varepsilon>0$.
Theorem~\ref{thm:incidence} provides the best known lower bound for this problem, as every incidence graph $G(P,\mathcal{H})$ of $P$ and $\mathcal{H}$ in $\mathbb{R}^d$ is a semi-algebraic graph in $(\mathbb{R}^d,\mathbb{R}^d)$ with bounded description complexity.

\bibliographystyle{plain}
\bibliography{bibliography}

\begin{thebibliography}{10}

\bibitem{acke15}
Eyal Ackerman.
\newblock On topological graphs with at most four crossings per edge.
\newblock Submitted, preliminary version:
  \url{http://arxiv.org/abs/1509.01932}, 2015.

\bibitem{apfSha07}
Roel Apfelbaum and Micha Sharir.
\newblock Large complete bipartite subgraphs in incidence graphs of points and
  hyperplanes.
\newblock {\em SIAM J. Discrete Math.}, 21(3):707--725, 2007.

\bibitem{bana93}
Wojciech Banaszczyk.
\newblock New bounds in some transference theorems in the geometry of numbers.
\newblock {\em Math. Ann.}, 296(4):625--635, 1993.

\bibitem{bhpt01}
Imre B{\'a}r{\'a}ny, Gergely Harcos, J{\'a}nos Pach, and G{\'a}bor Tardos.
\newblock Covering lattice points by subspaces.
\newblock {\em Period. Math. Hungar.}, 43(1--2):93--103, 2001.

\bibitem{braKna03}
Peter Brass and Christian Knauer.
\newblock On counting point-hyperplane incidences.
\newblock {\em Comput. Geom.}, 25(1--2):13--20, 2003.

\bibitem{braMoPa05}
Peter Brass, William Moser, and J{\'a}nos Pach.
\newblock {\em Research problems in discrete geometry}.
\newblock Springer, New York, 2005.

\bibitem{chaze93}
Bernard Chazelle.
\newblock Cutting hyperplanes for {D}ivide-and-{C}onquer.
\newblock {\em Discrete Comput. Geom.}, 9(2):145--158, 1993.

\bibitem{erd46}
Paul Erd{\H o}s.
\newblock On sets of distances of $n$ points.
\newblock {\em Amer. Math. Monthly}, 53:248--250, 1946.

\bibitem{erickson96}
Jeff Erickson.
\newblock New lower bounds for hopcroft's problem.
\newblock {\em Discrete Comput. Geom.}, 16(4):389--418, 1996.

\bibitem{fpssz14}
Jacob Fox, J{\'a}nos Pach, Adam Sheffer, and Andrew Suk.
\newblock A semi-algebraic version of {Z}arankiewicz's problem.
\newblock {\em J. Eur. Math. Soc. (JEMS)}, 19(6):1785--1810, 2017.

\bibitem{hagRu90}
Torben Hagerup and Christine R{\" u}b.
\newblock A guided tour of {C}hernoff bounds.
\newblock {\em Inform. Process. Lett.}, 33(6):305--308, 1990.

\bibitem{henk02}
Martin Henk.
\newblock Successive minima and lattice points.
\newblock {\em Rend. Circ. Mat. Palermo (2) Suppl.}, 70(I):377--384, 2002.

\bibitem{john48}
Fritz John.
\newblock Extremum problems with inequalities as subsidiary conditions.
\newblock In {\em Studies and Essays, presented to R. Courant on his 60th
  birthday, January 8, 1948}, pages 187--204. Interscience Publ., New York,
  1948.

\bibitem{lefm12}
Hanno Lefmann.
\newblock Extensions of the {N}o-{T}hree-{I}n-{L}ine {P}roblem.
\newblock Submitted, preliminary version:
  \url{www.tu-chemnitz.de/informatik/ThIS/downloads/publications/lefmann_no_three_submitted.pdf},
  2012.

\bibitem{mahl39}
Kurt Mahler.
\newblock Ein {{\" U}}bertragungsprinzip f{\" u}r konvexe {K}{\" o}rper.
\newblock {\em {\v C}asopis P{\v e}st. Mat. Fys.}, 68:93--102, 1939.

\bibitem{mat02}
Ji\v{r}\'{\i} Matou\v{s}ek.
\newblock {\em Lectures on Discrete Geometry}, volume 212 of {\em Graduate
  Texts in Mathematics}.
\newblock Springer-Verlag, New York, 2002.

\bibitem{mink10}
Hermann Minkowski.
\newblock {\em Geometrie der {Z}ahlen}.
\newblock Leipzig, Teubner, 1910.

\bibitem{pachToth97}
J\'{a}nos Pach and G\'{e}za T\'{o}th.
\newblock Graphs drawn with few crossings per edge.
\newblock {\em Combinatorica}, 17:427--439, 1997.

\bibitem{roth51}
Klaus~Friedrich Roth.
\newblock On a problem of {H}eilbronn.
\newblock {\em J. London Math. Soc.}, 26:198--204, 1951.

\bibitem{shef15}
Adam Sheffer.
\newblock Lower bounds for incidences with hypersurfaces.
\newblock {\em Discrete Anal.}, 2016.
\newblock Paper No. 16, 14.

\bibitem{sieChan89}
Carl~Ludwig Siegel and Komaravolu Chandrasekharan.
\newblock {\em Lectures on the geometry of numbers}.
\newblock Springer-Verlag, Berlin, 1989.

\bibitem{szeTro83}
Endre Szemer{\' e}di and William~T. Trotter~Jr.
\newblock Extremal problems in discrete geometry.
\newblock {\em Combinatorica}, 3(3--4):381--392, 1983.

\end{thebibliography}

\end{document}